\documentclass[11pt]{amsart}
\usepackage[utf8]{inputenc}

\usepackage{amsmath,amsthm,verbatim,amssymb,amsfonts,amscd, graphicx}
\usepackage{xcolor}
\usepackage{graphics}
\usepackage{mathtools}
\usepackage{microtype}
\usepackage{euscript, enumerate}
\usepackage{url,hyperref}
\mathtoolsset{showonlyrefs}

\makeatletter
\def\blfootnote{\xdef\@thefnmark{}\@footnotetext}
\makeatother

\newcommand{\p}{\partial}

\newcommand{\la}{\langle}
\newcommand{\ra}{\rangle}

\newcommand{\e}{\epsilon}
\newcommand{\eps}{\epsilon}

\newcommand{\be}{\begin{equation}}
\newcommand{\ba}{\begin{aligned}}
\newcommand{\bee}{\begin{equation*}}
\newcommand{\ee}{\end{equation}}
\newcommand{\ea}{\end{aligned}}
\newcommand{\eee}{\end{equation*}}
\newcommand{\bea}{\begin{equation} \begin{aligned} }
\newcommand{\eea}{\end{aligned}\end{equation}}
\newcommand{\R}{{\mathbf R}}
\newcommand{\normm}{ } 

\DeclareMathOperator{\spt}{spt}
\DeclareMathOperator{\dist}{dist}

\newcommand{\grad}{\nabla}
\newcommand{\laplace}{\Delta}
\newcommand{\N}{\mathbf{N}}
\renewcommand{\div}{\grad\cdot}
\newcommand{\U}{\mathcal{U}}

\newcommand{\red}[1]{{\textcolor{red}{#1}}} 
\newcommand{\blue}[1]{{\textcolor{blue}{#1}}} 
\newcommand{\black}[1]{{\textcolor{black}{#1}}} 

\theoremstyle{plain}
\newtheorem{theorem}{Theorem}[section]

\newtheorem{corollary}[theorem]{Corollary}

\newtheorem{lemma}[theorem]{Lemma}
\newtheorem{proposition}[theorem]{Proposition}
\newtheorem{prop}[theorem]{Proposition}

\theoremstyle{remark}

\newtheorem{remark}[theorem]{Remark}

\theoremstyle{definition}
\newtheorem{definition}[theorem]{Definition}

\numberwithin{equation}{section}

\begin{document}
\title[Fast diffusion asymptotics on a bounded domain]{Asymptotics near extinction for nonlinear fast diffusion on a bounded domain}
\thanks{
BC thanks the University of Toronto, his affiliation at the time this research was commenced.  BC's research was partially supported by National Research Foundation of Korea grant No. 2022R1C1C1013511, POSTECH Basic Science Research Institute grant No. 2021R1A6A1A10042944, and POSCO Science Fellowship.  RJM's research was supported in part by the Canada Research Chairs Program and
Natural Sciences and Engineering Research Council of Canada Grant RGPIN 2020-04162. CS's work is funded by the Deutsche Forschungsgemeinschaft (DFG, German Research Foundation) under Germany's Excellence Strategy EXC 2044 --390685587, Mathematics M\"unster: Dynamics--Geometry--Structure.
\copyright 2022 by the authors.
}

\author{Beomjun Choi}
\address{BC: Department of Mathematics, POSTECH, Pohang, Gyeongbuk, South Korea}
\email{bchoi@postech.ac.kr}

\author{Robert J. McCann}
\address{RJM: Department of Mathematics, University of Toronto, Toronto ON Canada}
\email{mccann@math.toronto.edu}

\author{Christian Seis}
\address{CS: Institut f\"ur Analysis und Numerik, Westf\"alische Wilhelms-Universit\"at M\"unster, M\"unster, Germany}
\email{seis@wwu.de}

\begin{abstract}
On a smooth bounded Euclidean domain, Sobolev-subcritical fast diffusion with vanishing boundary trace is known to lead to finite-time extinction, with a vanishing profile selected by the initial datum. In rescaled variables, we quantify the rate of convergence to this profile uniformly in relative error, showing the rate is either exponentially fast (with a rate constant predicted by the spectral gap), or algebraically slow (which is only possible in the presence of non-integrable zero modes). In the first case, the nonlinear dynamics are well-approximated by exponentially decaying eigenmodes up to at least twice the gap; this refines and confirms a 1980 conjecture of Berryman and Holland. We also improve on a result of Bonforte and Figalli, by providing a new and simpler approach which is able to accommodate the presence of zero modes, such as those that occur when the vanishing profile fails to be isolated (and possibly belongs to a continuum of such profiles).
\end{abstract}
\maketitle

  \section{Introduction}\blfootnote{To appear in Archive for Rational Mechanics and Analysis.}

Setting $m_q := 1 -\frac{2}{n+q}$ \cite{DenzlerKochMcCann15}, consider the fast diffusion equation with the exponent $0<m \in ({m_{1-\frac n2}},1)$, integrable non-negative initial data, and Dirichlet boundary conditions, 
on a smooth bounded domain  $\Omega \subset \R^n$:
\bea \label{eq-fastdiffusion}
w_\tau = \Delta (w^m)  \text{ on } \Omega
\\ w=0 \text{ on } \p \Omega .
\eea
This equation models heat flow in a material whose thermal conductivity $m w^{m-1}$ depends inversely on its local temperature $w$.  {With $m=1/2$ it has also been used to model the 
diffusion of plasma ions across a magnetic field in simulations \cite{OkudaDawson73} and experiments \cite{TamanoPraterOhkawa73}.
}
{For such an initial value problem,} it is known that the {vanishing Dirichlet boundary conditions drive the solution $w$ to} become extinct in finite time, e.g., \cite{Sabinina62,Sabinina65,BerrymanHolland80}. {To understand the vanishing profile, rescale} around the extinction time $T>0$,
\bea
{ w(x,\tau)= \left((1-m)(T-\tau)\right)^{\frac{1}{1-m}} v^{\frac1m}(x,t), \quad t= \frac{m}{1-m} \ln \frac{T}{T-\tau}}
 \eea
  to obtain an equation for $v(x,t)$ with $p=1/m$:
\bea  \label{eq-U}
\frac{\p\,}{\p t} 
\left(\frac{v^p}{p}\right)
= \Delta v + v^p  \text{ on } \Omega \\
v=0 \text{ on }\p\Omega.
\eea
 The rescaled solution $v$ on $\Omega$ is known to converge (both subsequentially \cite{BerrymanHolland80} and sequentially \cite{FeireislSimondon00})
 as $t\to \infty$ to some unique profile $V(x)$ 
 selected by the initial dat{um};
 moreover, $V$ is a positive solution to the stationary elliptic problem 
 
 \bea \label{eq-V}  \Delta V +  V^p=0 \text{ on } \Omega  
 \\ V=0 \text{ on } \p \Omega.
 \eea
Solutions to \eqref{eq-V} represent critical points of the Lyapunov functional {\cite{BerrymanHolland78}}
\begin{equation}\label{Lyapunov}
E(v)=\frac12\|\nabla v\|^{ 2}_{L^2(\Omega)} -\frac 1{p+1} \|v\|_{L^{p+1}(\Omega)}^{p+1}
\end{equation}
for the dynamics \eqref{eq-U}.  
 Sobolev subcriticality $p^{-1} > {m_{1-\frac n2}} =\frac{n-2}{n+2}$ implies 
 {$L^{p+1}$-coercivity of the energy on the $L^{p+1}$-unit sphere,  from which    the  existence of positive steady states \eqref{eq-V}
 had earlier been derived by Berger \cite{Berger77}. 
  Brezis and Nirenberg showed such solutions need not be unique however \cite{BrezisNirenberg83}; 
  other uniqueness and nonuniqueness results concerning
   positive solutions in specific domain geometries may be found in the works of 
   Gidas-Ni-Nirenberg \cite{GidasNiNirenberg79} on the ball,  
    Dancer \cite{Dancer88} \cite{Dancer90} on connected approximations to disjoint unions of balls, 
    Damascelli-Grossi-Pacella \cite{DamascelliGrossiPacella99} on domains with symmetry, 
   Zou \cite{Zou94} on rough balls,       
   Akagi-Kajikiya \cite{AkagiKajikiya14} who used instability to show that uniqueness fails on thin annuli, {so solutions and their rotations form continuous families,}  and Akagi \cite{Akagi16} who showed the stability of energy minimizers. 
   On the other hand,
   Feireisl-Simondon \cite{FeireislSimondon00} showed that the evolution \eqref{eq-U} selects and converges to one of the positive solutions $V$ of \eqref{eq-V} (which may depend on the initial dat{um}) as $t\to \infty$. They did not give a result on the convergence rate. Later Bonforte, Grillo and V\'azquez \cite{BonforteGrilloVazquez12} showed convergence in relative error $h:=\frac{v-V}{V}$, i.e. 
 \bea \label{vre} 
 \lim_{t\to \infty} \left \Vert \frac{v(x,t)}{V(x)} -1 \right \Vert _{L^\infty(\Omega)} = 0,  \eea 
{and provided an exponential rate of convergence in entropy sense,
under a non-degeneracy condition which they were able to verify for $m$ close to $1$.}

It has been a problem of considerable interest (a) to quantify the rate of convergence {unconditionally, and 
(b) to predict the higher-order asymptotics of the relative error $h$.}
{In contrast to the analogous questions set on the full space $\Omega = \R^n$, resolved in \cite{DenzlerKochMcCann15} and its references \cite{BlanchetBonforteDolbeaultGrilloVazquez09} \cite{CarrilloToscani00} \cite{CarrilloVazquez03} \cite{DelPinoDolbeault02} \cite{KimMcCann06} \cite{Otto01}, this challenge is compounded by the fact that the linearized problem has unstable modes ({including those} corresponding to $\tau$-translations in the original variables, which blow up at different times $T$),  and can also have zero modes, including modes called {\em integrable} that arise e.g.~for reasons of symmetry,  as for the thin annuli mentioned above.
 Recently,  Bonforte and Figalli overcame {some of} these challenges to solve 
 (a) for $C^{2,\alpha}$-generic smooth domains including the ball \cite{bonforte2021sharp}.  In a suitable Hilbert space, they show the linearized evolution of $h$ is generated by a self-adjoint operator possessing a complete basis of eigenfunctions.
 To summarize their findings: the unstable (negative) modes cannot be active due to Feireisl and Simondon's convergence, and neutral (zero) modes are absent on generic domains \cite{SautTemam79},  
 in which case they show that the relative error decays {uniformly} with an exponential rate $\lambda$ no smaller than the 
 first positive eigenvalue.  {On arbitrary smooth domains however, such a result eludes their techniques,  which rely on the kernel of the linearized operator being trivial, and hence the limiting profile being isolated.
 {A simpler derivation of the rate of exponential convergence was subsequently obtained under the same restriction by Akagi \cite{Akagi21+}; although he expresses convergence in terms of an energy rather than the entropy or relative error,
 these quantities can be compared using the boundary regularity theory of \cite{JinXiong22+}.
Finally, Jin and Xiong showed unconditionally that the rate of convergence is at least algebraic in $t$ \cite{JinXiong20+},
but with a power that is not explicit. In the present manuscript, we bridge this wide gap (between exponential upper and algebraic lower bounds on the rate) by developing a new approach which yields that}
  $\|h\|_{L^\infty}$ either decays exponentially with rate $\lambda$ or else decays algebraically at a rate $1/t$ or slower. 
In the second case, {not only must
zero modes be present,  but they must be non-integrable in a sense made precise below.
Moreover, when the decay rate is exponential,} 
we address (b) by showing that the longtime asymptotics of the nonlinear problem are described by the linearized dynamics up to the error $e^{-2\lambda t}$ produced by quadratic corrections. {This refines and confirms a conjecture made for the case $n=1$ by Berryman and Holland \cite{BerrymanHolland80}.}

Besides being more powerful, our approach is also simpler than Bonforte and Figalli's. Instead of augmenting Del Pino and Dolbeault's nonlinear entropy method \cite{DelPinoDolbeault02} with approximate orthogonality conditions, we rely
on an ODE lemma of Merle and Zaag \cite{MerleZaag98},  which implies that the dynamics are eventually dominated either by  stable or neutral modes. 
We must also control the nonlinearity  by adapting parabolic regularity estimates to the geometry of the steady state around the domain boundary. Delicately matching these estimates to the ODE argument allows us to estimate the rate of convergence via the dichotomy described above.   From there we use Hilbert projection techniques
to get an asymptotic expansion up to an $e^{-2\lambda t}$ error.  The latter resembles Denzler, Koch and McCann's treatment of higher-order asymptotics in the narrower range
$m_0=1-\frac2n<m<1$ of 
evolutions on the unbounded domain $\Omega=\R^n$ --- though the linearized analysis around the selfsimilar spreading solution 
described in \cite{DenzlerKochMcCann15} \cite{DenzlerKochMcCann16} and their references is more subtle than the present problem,
and our rate estimate does not require us to establish differentiable dependence of the flow on initial conditions.
{On the other hand,  the whole space problem is not plagued by the multiplicity and continua of limiting profiles that we presently face.}
In the current setting, finer aspects of the dynamics (beyond
rate $2\lambda$) may conceivably be described by constructing invariant manifolds, but we defer the exploration of this possibility to future research.  
In the porous medium regime (which refers to the complementary range of nonlinearities $m>1)$ on the whole space $\Omega=\R^n$,  such a construction was completed by one of 
us \cite{Seis14} \cite{Seis15+} following earlier work of Angenent \cite{Angenent88} for $n=1$ and Koch \cite{Koch99} for $n>1$.

The remainder of this manuscript is structured as follows.  In the subsequent Section \ref{S1a}, we rewrite the problem in terms of the relative error, for which most of our analysis is conducted, summarize the spectral theory and introduce some notation. Along with the necessary terminology, Section \ref{S2} states our two main dichotomy results. After recalling variants of the Merle-Zaag lemma \cite{MerleZaag98} due to K.~Choi with Haslhofer and Hershkovits on the one hand \cite{ChoiHaslhoferHershkovits18+} and with Sun on the other \cite{ChoiSun20+}, our first dichotomy --- separating fast and slow convergence --- is proved in Section \ref{S3},  apart from the parabolic regularity estimates which yield a quadratic bound in the relative error for the nonlinearity.  These are postponed to Section \ref{S4}.  The remaining dichotomy is established in Section \ref{S:second dichotomy}.

\section{Linearized dynamics and relative error}\label{S1a}
 In terms of the relative error $h:=\frac{v-V}{V} $, 
 the dynamics \eqref{eq-U} take the form
  \bea \label{eq-relativeerror}  \p_t h +  L_V h =  N (h),\eea 
 {where $L_V$ is the linear operator relative to $V$,}
 \begin{equation}\label{eqL}
 \begin{aligned}
 L_{V} h &= -\frac{1}{ V^{p}} \Delta (hV) - p h\\
& = -V^{1-p}\laplace h - 2V^{-p} \grad V\cdot \grad h - \left({p-1}\right)h\\
& = -V^{-1-p}\div(V^2 \grad h) - \left({p-1}\right)h,
 \end{aligned}
 \end{equation}
{and $N(h)$ is the nonlinearity, given by}
 \bea \label{eq-problemrelative}
 N (h) &= (1+h)^p - 1 - p h  + \big(1-(1+ h)^{p-1} \big)  \partial_t h 
 .\eea
Observe that $N(h) =   M_V(h)$ for any solution $h$ to \eqref{eq-relativeerror},
where
\begin{equation}
\label{32a}
 {M_V} (h) = \frac 1{(1+h)^{p-1}} \left((1+h)^p-1-ph\right) + \left(1- \frac1{(1+h)^{p-1}} \right)
  L_{V}h.
\end{equation}
Indeed, solving \eqref{eq-relativeerror} for $\partial_t h$, dividing by the prefactor and shifting the nonlinearities onto the right-hand side yields  $N(h) = M_V(h)$. 
This allows us to exchange temporal for spatial derivatives of $h$ in the nonlinearity. 
Since in most parts of the paper the reference stationary solution is fixed, we will often write $L=L_V$ for notational simplicity.

The  relative error  and the linear operator are best understood when analyzed in suitable weighted Lebesgue and Sobolev spaces. Given $\sigma>0$ and a positive solution $V$ to \eqref{eq-V}, the weighted inner product 
$$\la f,g\ra_{\sigma}= \int_\Omega fg V^\sigma dx
$$ makes
$
L^2_\sigma =L^2_\sigma(\Omega) :=  \{ f :\: \la f,f\ra_{\sigma} <\infty \}
$
into a Hilbert space. We will  occasionally  be concerned with more general Lebesgue spaces induced by the norm
\[
\|f\|_{L^q_{\sigma}} = \left( \int_{\Omega} |f|^q\, d\mu_{\sigma}\right)^{\frac1q},
\]
where  $d\mu_p(x):=V(x)^p dx$. Weighted (homegeneous) Sobolev spaces such as $\dot H^1_{\sigma}$ are defined analogously; c.f.~\eqref{eigenvalue energy} below.

Multiplication by $V$ acts as an isometry between $L^2_{p+1}$ and $L^2_{p-1}$.
Under this isometry,  the linear operator {$L_V + pI$} is unitarily equivalent to an operator $\tilde L = V \circ {(L_V + pI)} \circ V^{-1}$ with compact inverse on $L^2_{p-1}$,
whose spectral theory, subject to vanishing Dirichlet boundary conditions,  was elucidated by Bonforte and 
Figalli~\cite{bonforte2021sharp}: the corresponding operator
$L{=L_V}$ is a self-adjoint semibounded operator on $L^2_{p+1}$;the spectrum of $L$ is discrete and the eigenfunctions form a basis of $L^2_{p+1}$.  
{They are critical points for the restriction of the weighted Dirichlet energy
\begin{equation}\label{eigenvalue energy}
E_V(\phi)  =  \| \phi \|_{\dot H^1_2}^2 :=\int_\Omega |\nabla \phi|^2 V^2 dx
\end{equation}
to that subset of the $L_{p+1}$ unit-sphere for which the boundary trace of $\phi V$ vanishes}.
 Using two nonnegative integers $I$ and $K$, let us list the eigenvalues with repetition as 
\bea\lambda_{-I} \le \ldots \le \lambda_{-1} < 0= \lambda_{0}=  \ldots= \lambda_{K-1} < \lambda_{K} \le \ldots.     \eea


 \begin{itemize}
\item  The integer $I$ represents the dimension of the unstable modes of $L$ (and coincides with the Morse index of $E(v)$ from \eqref{Lyapunov} at $V$).
\item $K$ represents the dimension of kernel of $L$ and any corresponding eigenfunctions are called Jacobi fields. 
\item Let us call the eigenfunctions which correspond to $\lambda_{-I}$ to $\lambda_{-1}$ the {\em  unstable modes}, those corresponding to $\lambda_{0}$ to $\lambda_{K-1}$ the {\em neutral (or central) modes}, and the remaining eigenfunctions (starting with eigenvalue $\lambda_{K}$)the {\em stable modes.}  The corresponding eigenspaces will be denoted by $E_u$, $E_c$ and $E_s$, respectively. They are understood as subspaces of $L^2_{p+1}$, so that $L^2_{p+1} = E_u\oplus E_c \oplus E_s$.
\item $\lambda_{{-I}} =  1-p= 1-\frac1{m}$  and it is actually simple (a.k.a.~multiplicity $1$, so $\lambda_{-I}<\lambda_{1-I}$) 
with corresponding eigenfunction ${1}$ (called the ground state).  {In the original variables it corresponds to time translation of the solution;  the signs $I>0>\lambda_{-I}$ account for the fact that $\tau$-translations of a given solution disappear at different times $T$ hence diverge sharply from each other under the rescaling appropriate for one of them.}
\end{itemize}

 \textbf{Notation.} We finally comment briefly on some notation that we will frequently use throughout this work: We write $a\lesssim b$ if there is a constant $C$ such that $a\le C b$. The constant $C$ may depend on the limiting profile $V$, the domain, and other parameters such as $p=1/m$,  but this dependence
is continuous under small perturbations of $V$ in the relatively-uniform topology generated by
\eqref{relatively-uniform base} below,
hence $C$ may be regarded as being locally independent of $V$. We write $t \gg 1$ to indicate $t$ must be sufficiently large.}

\section{{Fast versus slow convergence dichotom{ies}}} 
\label{S2}

{Recall Bonforte, Grillo and V\'azquez \cite{BonforteGrilloVazquez12} showed  if $V(x)$ is the limit solution of $v(x,t)$ (see \cite{FeireislSimondon00}) then the relative error $h(t) = \frac{v(t)}V -1$ decays uniformly:
\bea
\label{eq-asymptoticrelative} \Vert h \Vert _{L^\infty(\Omega)} = o(1) \text{ as } t\to \infty. 
\eea 
In this section we describe two dichotomy theorems which establish that a spectral gap gives the sharp rate of exponential
 convergence,  unless the  
 linearized dynamics has a non-integrable kernel in the refined sense of Definition \ref{D3}. {When Definition \ref{D3} fails to be satisfied,} 
 we show the convergence occurs either exponentially at the rate of the spectral gap or no faster that $O(1/t)$.}
 
  \begin{theorem}[{First dichotomy} for asymptotic behavior]\label{thm-dichotomy}
   For $0<m\in {(m_{1-\frac n2},1)}=(\frac{n-2}{n+2},1)$, let $\Omega\subset \mathbb{R}^n$ be a smooth bounded domain and $v(x,t) {\ge 0}$ on $(x,t)\in \Omega \times [0,\infty)$ be a bounded solution to the {evolution} problem \eqref{eq-U}. Let $V(x)$ be the {classical solution to \eqref{eq-V}} satisfying \eqref{vre}. 
 Then 
{there exist positive constants
	$\epsilon=\epsilon(V,p)$ and $C=C(V,p)$ such that whenever $\|h(t)\|_{L^\infty} \le \epsilon$ for all $t \ge 0$} 
 exactly one of the following alternatives holds:  
 
   \begin{enumerate}
 	\item the relative error $h(t) := \frac{v(t)}{V} -1$ decays algebraically or slower
	\bea   \label{eq-slowdecay} 
	C\left \Vert h(t) \right\Vert_{L^\infty} \ge 
 \left \Vert h(t) \right\Vert_{L^2_{p+1}} 
\ge 
(Ct)^{-1}  
\qquad \forall\ t \gg 1; 
\eea   	
 	 \item the relative error $h$ decays exponentially or faster, 
	 	\bea \label{eq-fastdcay}
	C^{-1}\left \Vert h(t)
	\right\Vert_{L^2_{p+1}}\le \left \Vert h(t) 
	\right\Vert_{L^\infty }\le 
	C e^{-\lambda t} \|h(0)\|_{L^2_{p+1}} 
	\quad \forall\ t {\ge 1,
	}\eea 
 	{where $\lambda=\lambda_K>0$} is the first positive eigenvalue of linearized operator $ L$ in \eqref{eq-relativeerror}.  	
 	 
 	{Moreover, whenever \eqref{eq-fastdcay} holds for some $\lambda \ge \lambda_K$ and any $C=C(V,p,\lambda)$, if}   $J \in \N_0$ is given such   that $ \lambda_J< 2\lambda  $ and $\lambda_J<\lambda_{J+1} =: \mu$, and $\varphi_0, \varphi_{1}, \dots,\varphi_J$ are  corresponding eigenfunctions chosen in such a way that they can be extended to a complete $L^2_{p+1}$ orthonormal basis, then there are constants $C_i\in \mathbb{R}$  
	such that {$t \ge 1$ yields}  	
 	  \bea \label{eq-fastdcay'}\left \Vert 
	  h(t)- \sum_{i=0}^{J} C_{i}  e^{-\lambda_i t} \varphi_i\right \Vert_{L^2_{p+1}} 
	  {\le \tilde C\|h(0)\|_{L^2_{p+1}} \times} \begin{cases} e^{-2\lambda t} &  \mbox{if }\mu>2\lambda,\\ te^{-2\lambda t}&\mbox{if }\mu=2\lambda,\\
e^{-\mu t} &\mbox{if }\mu<2\lambda,\end{cases}
 \eea 
 {for some $\tilde C=\tilde C(V,p,\lambda,C)$}.
Finally, $C_i=0$ for any $\lambda_i \in [0,\lambda)$ so, in particular, for any $i< K$.
 \end{enumerate}
   \end{theorem}
 \begin{remark} 
  \begin{enumerate}
 \item The result of Bonforte-Figalli in \cite{bonforte2021sharp} shows \eqref{eq-fastdcay} 
 when $K=0$ (meaning $\ker\, L_V $ is trivial). This assumption prevents the first alternative \eqref{eq-slowdecay}.
 
 \item {As Bonforte and Figalli argued, we can relax the boundedness assumption on $v$ in Theorem \ref{thm-dichotomy} since initial conditions 
 \begin{equation}\label{initial constraints}
v_0 \in \begin{cases} 
\qquad\qquad L^1(\Omega) & {\rm if}\ m \in {(m_0,1)}=(1-\frac 2n,1) \\
\bigcup\limits_{q>\frac n2 (1-m)} L^q(\Omega) & {\rm if}\ m \in {(m_{1-\frac n2},m_0]} =(\frac{n-2}{n+2},1-\frac 2n]
\end{cases} 
\end{equation}
  lead to solutions which remain bounded after any short time \cite{DiBenedettoKwongVespri91}.}
  \item {Due to  Bonforte, Grillo and Vazquez result \eqref{eq-asymptoticrelative}, the smallness assumption on our initial data is no restriction either.}
  \item{{Notice the factor $t$ appearing in the bound \eqref{eq-fastdcay'} when $\mu=2\lambda$ accounts for the possibility of the kind of eigenvalue resonances described in \cite{Angenent88}.
  }}
 \end{enumerate}	
 \end{remark}
 
{As remarked above, triviality of the kernel, $K=0$, implies exponential decay.  However, we can also show that in certain situations,  exponential decay also occurs when $K>0$. 
For example, some kernel of $L$ can be obtained as a tangent variation among stationary states {(as for the non-rotationally symmetric limit states $V$ found by Akagi and Kajikiya on thin annuli \cite{AkagiKajikiya14}):} if there is a one-parameter family  $(V_s)_{s \in (-s_0,s_0)}$ of solutions to \eqref{eq-V} with $V_0=V$ and $\p_s V (0)\neq 0$,  then $L({\frac{\p_s V(0)}{V}}) = 0$.  If all zero modes are accounted for by such continuous symmetries in the sense of Definition~\ref{D3} below,  exponential decay occurs in spite of the fact that $K>0$.
To formulate this more precisely,  consider the whole family $S \subset H^{1}(\Omega)$ of weak solutions $V$ to the Dirichlet problem \eqref{eq-V}, i.e. Sobolev functions for which $\varphi \in C^1(\bar \Omega)$ implies
\begin{equation}\label{weak elliptic}
\int_\Omega \nabla \varphi \cdot \nabla V dx = \int_\Omega \varphi V^p dx. 
\end{equation} 
In Lemma \ref{L13}, we shall eventually infer $S \subset C^{3,\alpha}(\Omega)$ for some $\alpha \in (0,1)$ and that $V/W \in L^\infty(\Omega)$ for all $V,W \in S$. The latter implies that the sets\be \label{relatively-uniform base}
B_r(V) := \{ W \in S :\: \|\tfrac W  V -1 \|_\infty < r\}
\ee
form the {base} $\{B_r(V)\}_{r>0, V \in S}$ for a topology on $S$, called the {\em relatively-uniform} topology.
We call $V \in S$ an {\em ordinary limit} if $S$ forms a manifold of dimension $K=\dim \ker L_V$ near $V$,
which the error relative to $V$ embeds differentiably into $L_{p+1}$.  More explicitly, 
call $h$ a {\em stationary relative error} if it is a time-independent solution to \eqref{eq-relativeerror}, or equivalently} 
satisfies the nonlinear (singularly) elliptic equation
\begin{equation}\label{36a}
L_V h = M(h),\quad \mbox{where }M(h) =(1+h)^p-1-ph;
\end{equation}
{here $M(h)=N(h) = M_V(h)$ since $h$ is stationary.} \begin{definition}[Ordinary limit]\label{D3} We say that $V$ is an \emph{ordinary limit} if there exists a constant $\delta\in(0,1)$, an $L^2_{p+1}$-open neighborhood $\U$ of $0$ in $\ker L_V$ and a $C^1$ diffeomorphism 
\[
\Phi_V: (\U,\|\cdot \|_{L^2_{p+1}}) \to (\{ h\in L^\infty \cap \dot H^1_2  :\: L_Vh = M(h),\, \|h\|_{L^{\infty}}\le \delta\},\|\cdot\|_{L^2_{p+1}}),
    \]
with the properties that
\[
\Phi_V(0)=0,\quad (d\Phi_V)_0 = \mathrm{id}.
\]
In this definition, the identity map is the one on $\ker L_V$. 
\end{definition}

{
The point of this definition is the following theorem: 

\begin{theorem}[Second dichotomy: convergence to ordinary limits is fast]  
 \label{T3} {Under the hypotheses of Theorem \ref{thm-dichotomy},}
if a solution to \eqref{eq-U} converges to an ordinary limit $V \in S$, 
 then convergence takes place exponentially fast: i.e., \eqref{eq-fastdcay} holds.
 \end{theorem}

Here are some preliminary (rather formal) observations.

\begin{remark}[Ordinary limits have integrable kernels]
\label{R:ordinary limits have integrable kernels}
(1) If $h(s)$ is a smooth curve with $L_Vh(s) = M(h(s))$ and $h(0)=0$, then $\psi = \partial_s h(0) \in \ker L_V$.

(2) If $\psi \in \ker L_V$ consider $h(s)  = \Phi_V(s\psi)$ for $s$ small. Then
\[
L_V h(s) = M(h(s)) ,\quad h(0)=0,\quad \partial_s h(0)=\psi.
\]

\end{remark}

This remark shows the kernel of $L_V$ satisfies the next definition if $V \in S$ is an ordinary limit.

\begin{definition}[{Integrable kernel}]\label{D:integrable kernel}
The kernel of $L_V$ is called {\em integrable} if for each $\varphi \in \ker L_V$, there is a one-parameter family $\{V_s\}_{s\in(-\e,\e)} \subset S$ of solutions to \eqref{eq-V}, with $V_0=V$ and ${\frac{\p_s V(0)}{V}} =\varphi$.
\end{definition}

In the context of minimal surfaces and {geometric evolution equations},  analogous concepts of kernel integrability date 
back at least to Allard-Almgren \cite{AA} and Simon \cite{MR821971} respectively.
 Simon used analyticity to show a converse to Remark~\ref{R:ordinary limits have integrable kernels}.  
Inspired by this,  it is natural to expect integrability of the kernel of $L_V$
to imply that $V$ is an ordinary limit --- though we have not verified this in the present context.
{Note that a limit $V$ can fail to be ordinary in one of three ways: either (a) $S$ can fail to be a manifold nearby;
or (b) $S$ can be a manifold locally which the relative error fails to embed differentiably into $L^2(V^{p+1})$; 
or (c) locally $S$ can be a differentiably embedded manifold whose dimension is strictly less than that of $\ker L_V$.
It is natural to expect that limits which fail to be ordinary are rare:
To the extent that the behaviour of $S$ mimics the stratification and singularities of an analytic variety,  we imagine that (a) occurs only on a set having local codimension one in $S$.  Based on the regularity results we establish below, we are skeptical that (b) ever occurs.  And inspired by the genericity results of, e.g.,~Saut and Teman \cite{SautTemam79},  we conjecture that (c) is non-generic in the sense that there is a dense $G_\delta$ in $S$ on which the kernel of the linearized operator has the same dimension as $S$; i.e. the extra zero modes associated with non-ordinary limits are coincidences whose eigenvalues become non-zero upon perturbation. 
(A weaker conjecture is that a perturbation of $\Omega$ as well as $V$ preserving the dimension of $S$ is required to restore ordinariness.)  It is these expectations that motivate our choice of the term `ordinary'.
}

\label{W}
We finally translate the leading order asymptotics that we found for the relative error back to the  original fast diffusion equation \eqref{eq-fastdiffusion} and its convergence towards the separation-of-variables solution 
\begin{equation}
\label{54}
W(x,\tau) = ((1-m)(T-\tau))^{\frac1{1-m}} V(x)^{\frac1m}.
\end{equation}


{
\begin{theorem}[{Quantitative approach to self-similarity in original variables}]
\label{T2}
For $0<m\in (
{m_{1-\frac n2}},1)$, let $\Omega \subset \R^n$ be a bounded smooth domain and $w(x,\tau) {\ge 0}$ on $(x,\tau)\in \Omega\times [0,\infty)$ be a bounded solution {to the evolution \eqref{eq-fastdiffusion}} converging towards the separation-of-variables solution $W(x,\tau)$ given by \eqref{54}. 
{Then there exists $C=C(p,V)$ such that 
for $\|\frac{w^m}{W^m}-1\|_{L^\infty(\Omega \times[0,T])} $ and $T-\tau>0$ sufficiently small, 
after multiplication by $(T-\tau)^{-\frac 1{1-m}}$ exactly}
 one of the following alternatives holds:  

\begin{enumerate}
\item
the difference $w-W$ decays logarithmically or slower,
\[
{(T-\tau )}^{-\frac1{1-m}}\|w(\tau)-W(\tau) \|_{L^1_1} \ge \frac{1} { C \log^2({1-\frac\tau T})};
\]
\item or 
the difference $w-W$ decays algebraically or faster,
\[
{(T-\tau)}^{-\frac1{1-m}}\|w(\tau)-W(\tau) \|_{L^\infty} {\le  C {\|\frac{w^{m}(0)}{W^{m}(0)}-1 \|_{L^2_{p+1} }}} ({1-\frac\tau T})^{\frac{ m\lambda_{K}}{1-m}}.
\]
\end{enumerate}
\end{theorem}
}
{
\begin{remark}
Analogous estimates hold true for the relative error $\frac{w(\tau)}{W(\tau)}-1$.
\end{remark} 
}

\begin{proof}[Proof, Case 1]
{We start with the logarithmic decay estimate for which we assume that the first case in Theorem \ref{thm-dichotomy} applies. Then
\[
\frac1t \lesssim \|h(t)\|_{L^2_{p+1}} = \|v(t) - V\|_{L^2_{p-1}} \lesssim \|v(t)-V\|_{L^1_{p}}^{\frac12},
\]
where we have used in the second inequality that {$(v-V)^2 = (v+V)(v-V) = (h+2)V (v-V) \lesssim V |v-V|$ by the uniform boundedness of $h$.} Hence, by the definitions of $w$ and $W$ and the relation between $t$ and $\tau$, the latter estimate is equivalent to 
\[
\frac1{\log^2({1-\frac\tau T})} \lesssim (T-\tau)^{-\frac{m}{1-m}} \|w(\tau)^m - W^m(\tau)\|_{L^1_{p}}.
\]
We now use the elementary estimate $
|a^m-b^m| \le m\left(a^{m-1}+b^{m-1}\right)|a-b|$ for $a,b\in \R_+$ to bound
\begin{align*}
|w(\tau)^m-W(\tau)^m|V^{p-1} & \lesssim \left(w(\tau)^{m-1} + W(\tau)^{m-1}\right) V^{p-1} |w(\tau)-W(\tau)|\\
&\quad \lesssim (T-\tau)^{-1} \left(\frac{V^{p-1}}{v(t)^{p-1}}+ 1 \right)|w(\tau)-W(\tau)|.
\end{align*}
Since $V/v = 1/(h+1)$ is uniformly bounded by {hypothesis,}
the above analysis gives
\[
\frac1{\log^2({1-\frac\tau T})} \lesssim  (T-\tau)^{-\frac{1}{1-m}} \|w(\tau) - W(\tau)\|_{L^1_{1}},
\]
as desired.
}

{
\emph{Case 2.} We finally convert the exponential decay estimate on the relative error into an estimate for the fast diffusion equation \eqref{eq-fastdiffusion}. From the second case in Theorem \ref{thm-dichotomy}, the definitions of $W$ and $w$ and the relation between $t$ and $\tau$ we infer
\[ \|h(t)\|_{L^{\infty}} 
  \lesssim {\|\frac{w^m(0)}{W^m(0)}-1 \|_{L^2_{p+1}} } e^{-\lambda_{K} t} 
   \lesssim ({1-\frac\tau T})^{\frac{ m\lambda_{K}}{1-m}}
\]

We shall now invoke the elementary estimate $  |a-1|\lesssim |a^m-1| $ which holds true for $a$ close to $1$ to conclude a control on the relative error,
\[ \|\frac{w(\tau)}{W(\tau)}  -1\|_{L^{\infty}}
\lesssim {\|\frac{w^{m}(0)}{W^{m}(0)}-1 \|_{L^2_{p+1}} }( {1-\frac\tau T})^{\frac{ m\lambda_{K}}{1-m}}  
\]
which yields the desired statement in view of the scaling of $W$ and the boundedness of $V$,
\begin{align}
\|w(\tau)-W(\tau)  \|_{L^{\infty}}
& \le\|W(\tau)\|_{L^{\infty}} \|\frac{w(\tau)-W(\tau)}{W(\tau)}  \|_{L^{\infty}}
\\& \lesssim (T-\tau)^{\frac1{1-m}} \|\frac{w(\tau)}{W(\tau)}-1  \|_{L^{\infty}}.
\end{align}
This concludes the proof of the theorem.
}\end{proof}

\section{Proof {of first dichotomy} (apart from smoothing estimates)}
\label{S3}

We start with an $L^2$   semigroup estimate which is at the heart of our analysis.
\begin{lemma}[Energy growth {control} under nonlinear evolution]
\label{L17}
Let $h $ be a solution to the nonlinear equation {\eqref{eq-relativeerror}--\eqref{eq-problemrelative}} 
with  initial datum $h_{0}\in L^2_{p+1}$, and assume that $\|h\|_{L^{\infty}} \le \eps$ for some $\eps>0$. 
Then there exists $C>0$ such that for $\eps>0$ small enough and all $t>0$, the following holds:
\[
\|h(t)\|_{L^2_{p+1}}^2 +\int_0^t \|\grad h\|_{L^2_2}^2\, dt\le e^{Ct} \|h_0\|_{L^2_{p+1}}^2.
\]
\end{lemma}

\begin{proof}
It will be convenient to consider the purely spatial form \eqref{32a} of the nonlinearity $N(u) =M_V(u)$ when working with {\eqref{eq-relativeerror}--\eqref{eq-problemrelative}}. Then setting $d\mu_p(x):=V(x)^p dx$, testing {\eqref{eq-relativeerror}} with $h V^{p+1}$ and integrating by parts, we arrive at the identity
\begin{align*}
\MoveEqLeft\frac12\frac{d}{dt} \int h^2 \, d\mu_{p+1} + \int |\grad h|^2\, d\mu_2 \\
=& (p-1) \int h^2\, d\mu_{p+1} +  \int \frac{h}{(1+h)^{p-1}}\left((1+h)^p-1-ph\right)\, d\mu_{p+1} \\
& +\int \left(\left(\frac1{1+h}\right)^{p-1}-1 \right) |\grad h|^2\, d\mu_2 + (p-1) \int h \left(\frac1{1+h}\right)^{p} |\grad h|^2\,d\mu_2\\
& +(p-1) \int \left(\left(\frac1{1+h}\right)^{p-1}-1\right) h^2\, d\mu_{p+1}.
\end{align*}
Because $|h|\le \eps<1$, the right-hand side can be controlled by quadratic expressions, more precisely,
\[
\frac{d}{dt} \int h^2 \, d\mu_{p+1} + \int |\grad  h|^2 \, d\mu_2 \lesssim \int h^2\, d\mu_{p+1}+\eps \int |\grad  h|^2 \, d\mu_2 .
\]
If $\eps$ is sufficiently small, the gradient terms can be absorved into the right hand side, and the resulting estimate can be solved with the help of a standard Gronwall argument. This proves the lemma.
\end{proof}

The proof of our {first dichotomy}, Theorem \ref{thm-dichotomy}, is based on two main ingredients. On the one hand, our argument will rely on a fundamental dynamical systems result that is due to Merle and Zaag \cite{MerleZaag98},  and which in turn improves on an earlier related result by  Filippas and Kohn \cite{FilippasKohn92}. On the other hand, we have to exploit some smoothing properties of the parabolic equation.

The Merle--Zaag lemma is concerned with a system of weakly coupled first order ordinary differential equations featuring  stable, neutral and unstable solutions. It states that under the assumption that
 the unstable modes fail to grow, the long time asymptotics are dominated by precisely one of the other two modes.
 {This lemma provides an effective way to extract the {quantized behaviour of the solution prescribed by the discrete spectrum of its limit.} It plays a {pivotal}  role in recent progress on classifications of ancient solutions (solutions defined for $(-\infty,T]$) to parabolic equations \cite{angenent2020uniqueness, angenent2019unique, brendle2019uniqueness, choi2019ancient} and entire solutions to elliptic equations \cite{CCK} arising from geometry. In classifications of ancient flows, the lemma is applied backward in time (i.e.~$t=-s$). One advantage in backward problems is that there are only finitely many stable eigenfunctions (imagine, for instance, the laplacian $\Delta$ has only finitely many positive eigenfunctions); this makes the classifications possible. Meanwhile in the forward problem, there are infinitely many stable eigenfunctions. The lemma is used to investigate the asymptotic behavior of solutions.  
{To obtain the stated dependencies of the constants in our first dichotomy requires a refinement of the Merle-Zaag 
Lemma due to K.~Choi et al:}

\begin{lemma}[{Choi-Haslhofer-Hershkovits refinement \cite[Lemma 4.6]{ChoiHaslhoferHershkovits18+}}] \label{lem-MZODE3}
	Let $X(s)$, $Y(s)$, and $Z(s)$ be non-negative absolutely continuous functions on $[0,\infty)$ satisfying  $X+Y+Z >0$, 
	\begin{equation} \label{eq:mz.ode.system}
		\begin{aligned}
			\frac{dX}{ds} -X &\geq  -\eps(Y + Z), \\
			|\frac{dY}{ds}| &\leq \eps (X+Y+Z ), \mbox{and}\\
			\frac{dZ}{ds} + Z &\leq  \eps(X+Y)
		\end{aligned}
	\end{equation}
	for each $\e\in (0, \frac{1}{100})$ and a.e. $s\in  [s_0(\e),\infty)$. If
	\begin{equation} \label{eq:mz.ode.nonzero}
		\lim_{s\to \infty} (X + Y + Z)(s) = 0
	\end{equation}	
	then  $ X \le 2\eps (Y+Z)$ for $s\ge s_0(\e)$ and 
	either 
	\begin{equation} \label{eq:mz.ode.A}
X(s)+Z(s)= o(Y(s)) \text{ as } s\to \infty 
	\end{equation}or 
	\begin{equation} \label{eq:mz.ode.B}
	X(s)+Y(s)  \le 100 \e  Z(s)  \text{ for } s\ge s_0(\e).
	\end{equation}
	\end{lemma}
	
	\begin{proof}
	A proof is given in Appendix B of \cite{choi2019ancient}.
	\end{proof}

{{For later use {we also recall a quantitative adaptation of the Merle Zaag lemma} to a compact time interval proved by Choi and Sun in \cite{ChoiSun20+}.
In Section \ref{S:second dichotomy} their result} will be {motivated and} used to prove our second dichotomy.
}

\begin{lemma}[{Choi-Sun refinement \cite[Lemma B.2]{ChoiSun20+}}]
 \label{lem-choisun} Suppose $X(s)$, $Y(s)$, and $Z(s)$ are non-negative absolutely continuous functions on some interval $[-L,L]$ such that $0<X+Y+Z< \eta$ for some $\eta>0$. Suppose that there exist two constants $\sigma >0$ and $\Lambda>0$ such that
\bea \frac{dX}{ds} - \Lambda X  & \ge   - \sigma (Y+Z), \\
 |\frac{dY}{ds}| &\le   \sigma (X+Y+Z),\\
  \frac{dZ}{ds} + \Lambda Z  & \le  \sigma(X+Y) ,
  \eea 
for any $s\in[-L,L]$. Then there exists {$\sigma_0=\sigma_0(\Lambda)$} such that  if $0<\sigma<\sigma_0$ it holds 
\bea X+Z \le \frac{8\sigma}{\Lambda} Y + 4\eta e^{-\frac{\Lambda L}{4}} \mbox{ for  any } s\in[-L/2,L/2].\eea	
\end{lemma}
}

{The next proposition provides the crucial control that we need to estimate the nonlinear terms in the relative error dynamics \eqref{eq-relativeerror}  quadratically:}

\begin{prop}[Spatially uniform control of time derivatives]
\label{P1}
Let $k\in \N_0$ and $t>0$ fixed. Then if $\|h\|_{L^{\infty}}\le \eps$ with $\eps$ sufficiently small, there exists a constant $C=C(t,k{,m,V})$ such that
\[
\|\partial_t^{k}  h(t) \|_{L^{\infty}} \le C \|h_0\|_{L^2_{p+1}}.
\]
\end{prop}

We postpone the proof and a discussion of this proposition to the next subsection and show first how to deduce 
{our first dichotomy} from Lemmas \ref{L17}--\ref{lem-MZODE3} and Proposition~\ref{P1}.

\begin{proof} [Proof of Theorem \ref{thm-dichotomy}]
	{In this proof}, for the sake of convention, we omit the subscript $L^2_{p+1}=L^2(V^{p+1})$ and use
$ \Vert f \Vert: = \Vert f \Vert_{L^2_{p+1}} $. Positive constants $\e =\e (V,p)$ and $C=C(V,p)$ are not fixed yet, $\e$ may become smaller, and $C$ may become larger until they are fixed.

Let us denote by $P_s$, $P_c$ and $P_u$ the  orthogonal projections of $L^2_{p+1}$ onto the stable, center, and unstable  eigenspaces $E_s$, $E_c$ and $E_u$, respectively. Moreover, we write $h_s=P_sh$, $h_c=P_ch$, and $h_u=P_uh$ for the projected solutions.   A straightforward computation reveals that
\bea \label{eq-odelema0}\frac{d}{dt} \Vert h_u\Vert_{\normm} &\ge -\lambda_{{-1}} \Vert h_u \Vert _{\normm}- \Vert {N}(h) \Vert _{\normm},\\
\left| \frac{d}{dt} \Vert h_c \Vert_{\normm} \right| &\le  \Vert {N}(h) \Vert_{\normm}, \\
\frac{d}{dt}\Vert h_s \Vert_{\normm} &\le -\lambda_{{K}}  \Vert h_s \Vert_{\normm} +  \Vert {N}(h) \Vert_{\normm}    ,  \eea 
where we recall $\lambda_{{-1}}$ and $\lambda_{{K}}$ are the negative and positive eigenvalues of $L {=L_V}$ closest to zero. As an example, for the case of the evolution on the stable subspace, we observe that
\begin{equation}\label{projected evolution}
\partial_t h_s + L h_s = P_s{N}(h),
\end{equation}
and testing with $h_s V^{p+1}$ gives
\[
\frac12\frac{d}{dt} \|h_s\|^2 + \la h_s, Lh_s\ra  = \la h_s ,N(h)\ra.
\]
The third estimate in \eqref{eq-odelema0} now results from the lower bound on the stable eigenvalues and the Cauchy--Schwarz inequality. The first and the second estimates are derived analogously.

In order to estimate the nonlinearities  in \eqref{eq-odelema0}, we note that {for $|h|  \le \e $,} (provided that $\e$ is sufficiently small) there is $C_1=C_1(p)$ such that 
	\bea \label{51} |{N}(h)| &\le C_1 |h|  \left(\left|h\right| + \left|\partial_t h\right|\right) 
	\\&{\le C_1 |h|  \left(\left|h\right| + C\left\| h(t-1)
	\right\|\right)} 
		\eea 
by Taylor expansion followed by the smoothing estimates of Proposition~\ref{P1} for $t \ge 1$.
By virtue of Bonforte, Grillo and V\'azquez'  \cite{BonforteGrilloVazquez12} uniform bound on the relative error  \eqref{eq-asymptoticrelative}, for \emph{any} small ${\hat \e}>0$ there exists a time $t_0(\hat \e)$  such that \bea \label{50}\Vert {N}(h(t))\Vert_{\normm} \le {\hat \e} \Vert h(t)\Vert_{\normm} \text{ for any }t\ge t_0(\hat \e).\eea 
Plugging this estimate into \eqref{eq-odelema0}, we observe that if $h\neq0$, then $h_s$, $h_c$ and $h_u$ satisfy the hypothesis  Lemma   \ref{lem-MZODE3} {with $s=\lambda t$} where $\lambda := \frac{1}{2} \min (|\lambda_{-1}|, |\lambda_K|)$. Therefore, unless $h$ is stationary (and thus trivial), either
 \bea\label{eq-neutraldominate} \Vert h_u\Vert_{\normm}+\Vert h_s\Vert_{\normm} = o(\Vert h_c\Vert_{\normm}) \eea
or 
\bea\label{eq-stabledominate} \Vert h_u\Vert_{\normm}+\Vert h_c\Vert_{\normm} \le  \frac{100 \hat \e}{\lambda} \Vert h_s \Vert \text{ for } t\ge t_0(\hat \e)  .\eea Note in particular, the second alternative \eqref{eq-stabledominate} implies $\Vert h_u\Vert + \Vert h_c \Vert = o (\Vert h_s\Vert)$. We discuss the implications of each case individually, starting from the latter case \eqref{eq-stabledominate}.

 \noindent {\bf{Case 2.}}
First, by choosing  $\e =\e(V,p)$ sufficiently small, we may assume $\Vert h_u\Vert + \Vert h_c \Vert \le \frac 12 \Vert h_s \Vert $ for all $t \ge1$. Here, note that $t\ge 1$ is needed as Proposition~\ref{P1} is applied to estimate ${N}(h)$. Moreover, by the same argument we used to derive \eqref{50}, there is $C_2=C_2(V,p)$ such that 
\[ \Vert N(h(t))\Vert  \le C_2 \e \Vert h(t)\Vert \text{ for } t\ge 1. \]

Combining above,  the estimate of the nonlinearity \eqref{50} becomes
\[
\|N(h(t))\| \le C_2 {\eps} \|h(t)\|   \le 2C_2 \eps \|h_s(t)\|,
\]
for $t\ge 1$. The third estimate in   \eqref{eq-odelema0} thus turns into the differential inequality
\[
\frac{d}{dt} \|h_s\| \le - (\lambda_{K}-2C_2\eps) \|h_s\| ,
\]
for $t\ge 1$, which yields  via a Gronwall argument 
 \[
 \|h_s(t)\| \le  e^{-(\lambda_{K}-2C_2\eps) (t-1) } \|h_s(1)\| .
 \]By another   application of \eqref{eq-stabledominate} and the semigroup estimate from Lemma \ref{L17}, this estimate can be translated into the full solution
 \begin{equation}\label{52}
 \|h(t)\| \lesssim e^{-(\lambda_{K}-2C_2 \eps) t} \|h_0\|.
 \end{equation}

To {improve} this inequality by removing the $\eps$, fix $\eps$ small so that $ 2C_ 2 \e  <{\frac 12}\lambda_{K}$, and refine the estimate of the nonlinearity with the help of the quadratic bound \eqref{51}, the smoothing estimates from Proposition \ref{P1} and the previous estimate \eqref{52},
 \begin{align*}
 \|N(h(t))\| & \lesssim \left(\|h(t)\|_{L^{\infty}} + \|\partial_t h(t)\|_{L^{\infty}}\right)\|h(t)\| \\
 & \lesssim \|h(t-1)\|\|h(t)\| \\
 & \lesssim e^{-2(\lambda_{K}-2C_2 \eps) t} \|h_0\|^2,
 \end{align*}
 for $t\ge1$. Substitution into the third estimate of \eqref{eq-odelema0} finally establishes a differential inequality, which yields
  \begin{equation}\label{53}
 \|h(t)\| \lesssim e^{-\lambda_{K} t} \|h_0\|,
 \end{equation}
 via \eqref{eq-stabledominate}, for all $t \ge 0$ {(recalling $\|h_1\| \lesssim \|h_0\|$ from Lemma \ref{L17})}. By another application of Proposition~\ref{P1}, this estimate can be upgraded towards a uniform convergence result,
  \bea  \|h(t)\|_{L^{\infty}} \lesssim e^{-\lambda_{K} t} \|h_0\|, \eea for $t\ge1$.  
Having established \eqref{eq-fastdcay}, it remains to identify the leading order expansion \eqref{eq-fastdcay'}.
Let us begin by noting that whenever \eqref{eq-fastdcay} holds for some $\lambda \ge \lambda_K$
{and any $C$,} repeating the {previous estimate on the nonlinearity} 
yields the improvement
\begin{equation}\label{52a}
{ \|N(h(t))\| {\le \tilde C} e^{-2\lambda t} \|h_0\|^2}
 \end{equation} 
{for $t \ge 1$, where $\tilde C= \tilde C(V,p,\lambda, C)$.}
Since the nonlinearity is quadratic,  modes corresponding to eigenvalues in the interval $[0,2\lambda)$   are accessible via a simple ODE argument. For any $i$, the projection $y_i = \langle h,\varphi_i\rangle$ satisfies the differential inequality
 \begin{equation}\label{52b} \left|\frac{d}{dt} y_i+ \lambda_i y_i \right| \le \Vert {N}(h) \Vert_{\normm} ,
 \end{equation}
which can be rewritten as
\[
\left |\frac{d}{dt}(e^{\lambda_it}y_i)\right| \lesssim e^{-(2\lambda-\lambda_i)t},
\]
as a consequence of \eqref{52a}. Integration in time thus yields for any  $T\ge t {\ge 1}$ that
\[
\left| e^{\lambda_i t}y_i(t) - e^{\lambda_i T}y_i(T)\right| \lesssim e^{-(2\lambda- \lambda_i)t},
\]
where we have used the fact that $\lambda_i<2\lambda$. This bounds implies that $T \mapsto e^{\lambda_i T }y_i(T)$ is a Cauchy sequence, so that, sending $T\to\infty$, we find 
 \bea \label{eq-yi} \langle h(t),\varphi_i\rangle = y_i(t) = C_i e^{-\lambda_i t} + O(e^{-2\lambda t} ) \eea 
for some $C_i\in \R$.

We now estimate with the help of the decomposition $h = h_s+h_c+h_u$, the triangle inequality and the fact that the eigenfunctions $\varphi_{0},\dots,\varphi_J$ are orthonormal that
\begin{equation}\label{52f}
\|h - \sum_{i=0}^J C_i e^{-\lambda_i t} \varphi_i\| \le \sum_{i= 0}^J |y_i - C_i e^{-t\lambda_i}| + \|h_s - \sum_{i=0}^J y_i \varphi_i\| + \|h_c\| + \|h_u\|.
\end{equation}
We have just seen that the first term on the right-hand side is {$\lesssim e^{-2\lambda t}$.} 
{Recalling the spectral decomposition of $L$ and \eqref{projected evolution}}, 
the next contribution, $z=\|h_s - \sum_{i=1}^J y_i \varphi_i\|$ satisfies the differential inequality
 \begin{equation}\label{52c}
 \frac{d}{dt} z + \mu z \le   \Vert {N}(h) \Vert_{\normm}.
 \end{equation}
 Similarly to the argumentation above, we rewrite \eqref{52c} as
\[
\frac{d}{dt}(e^{\mu t}z) \lesssim e^{-(2\lambda-\mu)t},
\]
and deduce that
\begin{equation}\label{52d}
z(t)  \lesssim \begin{cases} e^{-2\lambda t} &  \mbox{if }\mu>2\lambda,\\ te^{-2\lambda t}&\mbox{if }\mu=2\lambda,\\
e^{-\mu t} &\mbox{if }\mu<2\lambda,\end{cases}
\end{equation}
via integration.

Finally, regarding the remaining terms in \eqref{52f}, the first two estimates in \eqref{eq-odelema0} imply that
\[
\frac{d}{dt} \|h_c\| +\frac{d}{dt} \|h_u\| \gtrsim  - e^{-2\lambda t},
\]
and thus, an integration from $t$ to $\infty$ gives, thanks to the fact that $\Vert h_c\Vert_{\normm} +\Vert h_u\Vert_{\normm} \to 0$ as $t\to \infty$,
 \bea  \label{eq-u+0}\Vert h_c(t)\Vert_{\normm} +\Vert h_u(t)\Vert_{\normm} \lesssim  e^{-2\lambda t}  .\eea 

 Therefore, substituting     \eqref{eq-yi},  \eqref{52d}, and \eqref{eq-u+0} into \eqref{52f}, we deduce that
 \bea \left \Vert h(t)- \sum_{i=0}^{J} C_{i} \varphi_i e^{-\lambda_i t} \right \Vert_{\normm}  \lesssim \begin{cases} e^{-2\lambda t} &  \mbox{if }\mu>2\lambda,\\ te^{-2\lambda t}&\mbox{if }\mu=2\lambda,\\
e^{-\mu t} &\mbox{if }\mu<2\lambda\end{cases}
 \eea 
as desired.  
{If $C_i\neq0$ for some $\lambda_i \in [0,\lambda)$, then \eqref{eq-fastdcay'} contradicts \eqref{eq-fastdcay}; thus $\lambda_i \ge \lambda$ as asserted.}

 \noindent {\bf{Case 1.}} In  case the neutral modes are dominating \eqref{eq-neutraldominate}, an estimate analogously to the one in Case 2 above gives rise to the bound
 \bea \Vert {N}(h(t))\Vert_{\normm} \le 2C_2\eps \|h_c(t)\|, \eea 
provided that $t\ge t_c$ for some $t_c$ large enough. 
 By plugging this into the middle equation of \eqref{eq-odelema0}, we find 
 \[
 \frac{d}{dt} \|h_c\| \ge -2C_2\eps\|h_c\|,
 \]
and thus, via \eqref{eq-neutraldominate}, \bea \|h(t+1)\|\ge \Vert h_c(t+1)\Vert_{\normm} \gtrsim \Vert h_c(t)\Vert_{\normm}\gtrsim \|h(t)\|,
   \eea
      for any $t\ge t_c$. With this information at hand, we may reconsider our previous bound on the nonlinearity. This time, making use of the pointwise estimate in \eqref{51} and the smoothing properties from Proposition \ref{P1}, we have
      \[
      \|N(h(t))\| \lesssim \|h(t-1)\|\|h(t)\| \lesssim \|h(t)\|^2 \lesssim \|h_c(t)\|^2,
      \]
      for any $t\ge t_c$, for some (possibly larger) $t_c$.  The second estimate in \eqref{eq-odelema0} now turns into the growth condition
      \[
      \frac{d}{dt} \|h_c(t)\| \gtrsim  - \|h_c(t)\|^2,
      \]
      which yields the lower bound
      \[
      \|h(t)\|\ge \|h_c(t)\| \gtrsim \frac1t,
      \]
      if $t\ge t_c$ for some $t_c$.

   This concludes the proof of Theorem \ref{thm-dichotomy}.
\end{proof}

\section{Smoothing estimates}
\label{S4}

{
In this section, we use parabolic regularity techniques to prove Proposition~\ref{P1}. We remark that optimal (boundary) regularity estimates were derived recently by Jin and Xiong for the rescaled solution $v$ of \eqref{eq-U} rather than the relative error $h$ \cite{jin2019optimal}. However, from Theorem 5.1 in their paper we easily infer that 
\begin{equation}
\label{200}
\partial_t^k h\in C^{0}((\tau,\infty)\times \Omega),
\end{equation}
for any $k\in \N$ and $\tau>0$. This insight will simplify the derivation of our regularity estimates substantially.
}

{
Smoothing features are  typical for parabolic equations and they remain true in the time and tangential directios for the singular parabolic equation under consideration,   see \eqref{200} and Corollary \ref{C12}. In transversal direction, regularity is limited  (if $p$ is not an integer) \cite{jin2019optimal}. Indeed,  simple scaling arguments for the elliptic problem suggest that $V(x)\sim a\dist(x,\partial \Omega) + b \dist(x,\partial\Omega)^{p+2}$ close to the boundary, and the same behavior can be expected for the parabolic problem \eqref{eq-U}.}

{
For deriving the smoothing estimates in Proposition~\ref{P1}, we notice that the leading order contribution  in the nonlinearity \eqref{eq-problemrelative} is of the order $|h||\partial_t h| \lesssim \eps |\partial_t h|$, and plays thus the role of a perturbation term in regularity estimates. Moreover, the equation is invariant under differentiation in time and in tangential coordinates near the domain boundary (at least with regard to the leading order contributions), and thus, regularity in these variables is propagated and even further increased by parabolicity. We deal with higher-order derivatives of the nonlinearity by applying suitable interpolations, so that eventually, derivatives of the nonlinearity will play the role of perturbations similarly to the nonlinearity itself as discussed above. Of course, smoothing proceeds instantaneously but not uniformly in time. For this reason, the estimates in or behind Proposition \ref{P1}, Equation \eqref{200} or Corollary \ref{C12} deteriorate  as $t\to 0$. 
}

Before addressing the dynamical problem, we need to estimate the first three derivatives for weak solutions 
of the nonlinear elliptic problem \eqref{eq-V}, starting from
the known result \eqref{24}. Later we'll see that 
higher-order tangential derivatives of this solution can also be estimated near the domain boundary.

\begin{lemma}[Regularity of asymptotic profile]
\label{L13}
For any $\alpha\in(0,1)$ with $\alpha \le p-1$, {if $V \in H^1(\Omega)$ satisfies \eqref{weak elliptic} for all $\varphi \in C^1(\bar \Omega)$, then} $V\in C^{3,\alpha}(\Omega)$ and  for any $x\in \Omega$, 
\begin{gather}
\dist(x,\partial \Omega) \lesssim V(x) \lesssim \dist(x,\partial \Omega) ,\label{24}\\
\mbox{\rm and}\ |\grad V(x)|, |\grad^2 V(x)|, |\grad^3 V(x)|\lesssim 1\label{25}
\end{gather}
{hold.} Furthermore, there exists an $r\gtrsim 1$ such that
\begin{equation}
|\grad V(x)| \gtrsim 1,\label{27}
\end{equation}
for any $x\in \Omega$ with $\dist(x,\partial \Omega)\le r$.
\end{lemma}

\begin{proof}
{For a fixed $V\in S$,}
the first estimate is established, for instance, in Theorem 1.1 in \cite{DiBenedettoKwongVespri91} (on the level of the evolutionary problem) or Theorem 5.9 in \cite{BonforteGrilloVazquez13};  {the form of \eqref{24} makes it clear that the constants depend continuously on $V$ in the relatively-uniform topology on $S$.}
The second and the third estimate follow from maximal regularity estimates and Sobolev embeddings. Indeed, since $V\in L^{\infty}(\Omega)$ thanks to \eqref{24}, we must have that $V\in W^{2,q}(\Omega)$ for any $q\in(1,\infty)$ based on Calder\'on--Zygmund estimates for the elliptic problem \eqref{eq-V}, {see, e.g., 
Chapter 11 in \cite{Krylov_Sob},} and thus $V\in C^{1,\alpha}(\bar \Omega)$ for any $\alpha\in(0,1)$ by Sobolev embeddings. It follows that $\partial_i V^p = pV^{p-1}\partial_i V\in C^{0,{\alpha}}(\bar \Omega)$ {provided that $\alpha\le p-1$}, and thus one spatial derivative of the equation shows  $V\in C^{3,{\alpha}}(\bar \Omega)$ by Schauder estimates, {see, e.g., Chapter 6 in \cite{GilbargTrudinger01}.}

The statement in \eqref{27} is a consequence {of the above}. Indeed, according to \eqref{24}, at the boundary $V$ grows linearly in the direction of the inner normal. By {the estimates \eqref{25}}, we must thus have \eqref{27} in a neighborhood of the boundary.  
\end{proof}

Our first step in the derivation of the smoothing estimates is a maximal regularity estimate for the linear{ization} of 
\eqref{eq-relativeerror}. More precisely,  we consider the inhomogeneous linear equation
\begin{equation}
\label{10}
\partial_t h  - V^{-1-p}\div (V^2 \grad h) = (p-1)h + f,
\end{equation}
with zero initial data. For  general initial data $h_0\in L^2_{p+1}$ and inhomogeneities $f\in L^1((0,T);L^2_{p+1})$, a solution is always understood in the weak sense.
A \emph{weak solution} of \eqref{10} refers to a function $h\in L^{\infty}((0,T);L^2_{p+1}) \cap L^2((0,T);\dot H^1_2)$ 
in the spaces provided by Lemma \ref{L17}
such that 
\begin{equation}
\label{10a}
\begin{aligned}
\MoveEqLeft 
-\iint_{(0,T)\times \Omega} \black{h} \partial_t \varphi\, d\mu_{p+1}dt + \iint_{(0,T)\times \Omega} \grad\varphi\cdot\grad h\, d\mu_2dt
\\
&  =   \iint_{(0,T)\times \Omega} ((p-1)h+f) \varphi d\mu_{p+1}dt + \int_{\Omega} \left.\varphi\right|_{t=0}h_0\, d\mu_{p+1},
\end{aligned}\end{equation}
for any $\varphi\in C_{c}^1([0,T)\times \bar \Omega)$ {of compact support}. {It should be stressed that we do not impose spatial boundary conditions on $h$ in the parabolic problem (5.5), which turns out to be well-posed (only) in this case. This is a consequence of the observation that by (formally) integrating by parts in the gradient term in the weak formulation \eqref{10a}, the boundary term vanishes thanks to the Dirichlet boundary conditions satisfied by $V$, see \eqref{eq-V}.}

Existence of weak solutions can be derived via standard methods, for instance, via Galerkin approximations based on an $L^2_{p+1}$ orthonormal basis consisting of eigenfunctions of the linear operator $L{=L_V}$.  Moreover, from standard energy estimates (derived similarly to those in Lemma \ref{L17}), we infer the uniqueness of weak solutions.

{What is a crucial tool in our theory is a maximal regularity estimate for the linear equation \eqref{10}, that we consider, for convenience, with $L^2_{2p}$ inhomogenity and zero initial data, see Proposition \ref{P5} below.

}
{In the interior of $\Omega$, maximal regularity for the parabolic problem \eqref{10} follows by standard theory, see, e.g., Chapter 7.1 in \cite{Evans}, because the diffusivity coefficients are strictly positive in the interior as a consequence of \eqref{24}.
We shall thus focus on the boundary from here on and we fix $x_0\in \partial\Omega$. Let $\eta$ denote a cut-off function on $\R^n$ interpolating smoothly between $\eta=1$ in $B_r(x_0)$ and $\eta = 0$ outside of $B_{2r}(x_0)$. }

{A short computation reveals that the localized solution $H =  \eta h$ satisfies the problem
\begin{equation}\label{10d}
 \partial_t H - V^{-1-p}\div(V^2\grad H) =    F+G,
\end{equation}
where $F$ and $G$ are given by
\[
F =  \eta f,\quad G =  -2V^{1-p} \grad\eta\cdot \grad  h - V^{1-p} h \laplace \eta   - 2 V^{-p} h \grad V \cdot \grad \eta .
\]
The following lemma guarantees that $G$ belongs to $L^2(L^2_{2p})$ provided that $h\in L^2(L^2_{p+1})\cap L^2( \dot H^1_2)$, which {is assumed for} our weak solutions.
}

\begin{lemma}[Weighted Poincar\'e / Hardy type inequality]
\label{L13b}
{All $h\in L^2_{p+1}\cap \dot H^1_2$ satisfy}
\[
\|h\|_{L^2} \lesssim \|h\|_{L^2_{p+1}} + \|\grad h\|_{L^2_2}.
\]
{In particular, it holds that 
\[
\|h\|_{L^2_{p+1}} \lesssim \|h\|_{L^2_{2p}} +  \|\grad h\|_{L^2_2}.
\]
}
\end{lemma}
The proof {of the first statement} is based on an interpolation argument and the properties of the limit~$V$. {The latter then follows via H\"older's and Young's estimate.}

\begin{proof}
{We start considering the first estimate.} {Because $C^{\infty}(\bar \Omega)$ is dense in $L^2_{p+1}\cap \dot H^1_2$, cf.~Lemma 2 in \cite{Seis14}, it is enough to establish the estimate for smooth functions. We first notice that an integration by parts and the defining properties of $V$ in \eqref{eq-V} yield
\begin{align*}
\int h^2 |\grad V|^2 \, dx &= - \int h^2 V \laplace V\, dx - 2 \int h\grad h\cdot  V\grad V\, dx\\
& = \int h^2 V^{p+1}\, dx  - 2 \int h\grad h\cdot  V\grad V\, dx.
\end{align*}
Making use of the elementary inequality $ab \le a^2 + b^2/4$ thus gives
\[
\| h |\grad V| \|_{L^2} \lesssim \|h\|_{L^2_{p+1}} + \|\grad h\|_{L^2_2}.
\]
The {first} statement of the lemma is now a consequence of the fact that $1\lesssim V^{p+1} + |\grad V|$, which holds true thanks to Lemma \ref{L13}.}

{For the second statement, we apply H\"older's inequality, and the previous bound to estimate
\[
\|h\|_{L^2_{p+1}} \le \|h\|_{L^2_{2p}}^{\frac{p+1}{2p}} \|h\|_{L^2}^{\frac{p-1}{2p}} \lesssim \|h\|_{L^2_{2p}}^{\frac{p+1}{2p}} \|h\|_{L^2_{p+1}}^{\frac{p-1}{2p}} + \|h\|_{L^2_{2p}}^{\frac{p+1}{2p}} \|\grad h\|_{L^2_{2}}^{\frac{p-1}{2p}}.
\]
The desired result is then a consequence of Young's inequality $ab \lesssim  a^q +  b^{q'}$ for any H\"older conjugates $q$ and $q'$.}
\end{proof}

{We shall now flatten the boundary. Upon a rotation of the coordinate system, we may assume that the boundary inside $B_r(x_0)$ can be written as a graph of a function $\gamma$, for instance,
\[
\Omega\cap B_{2r}(x_0)   = \left\{x=(x',x_n)\in B_{2r}(x_0):\: x_n>\gamma(x')\right\}.
\]
We set $\hat x = \phi(x) =(x',x_n-\gamma(x'))$, which defines a diffeomorphism in the support of $\eta$, and maps the boundary $\partial \Omega$ into the hypersurface $\R^{n-1}\times \{0\}$. In terms of $\hat H(t,\hat x) = H(t,x)$, $\hat V(\hat x) = V(x)$, $\hat F(t,\hat x) = F(t,x)$, and $\hat G (t,\hat x) = G(t,x)$ the localized  equation \eqref{10d} becomes
\begin{equation}
\label{10e}
 \partial_t \hat H- \hat V^{-1-p}\hat\grad\cdot(\hat V^2 A\hat\grad\hat H)   = \hat F + \hat G,
\end{equation}
where
\[
A = \left( \begin{array}{c|c} I & -\grad'\gamma \\\hline -(\grad'\gamma)^T & 1+|\grad'\gamma|^2\end{array}\right).
\]
Notice that the transformed equation \eqref{10e} has to be considered on the halfspace~$\R^n_+$. 
}

{The advantage of \eqref{10e} over the \eqref{10} is that in the new variables, the weight and its tangential derivatives can be estimated by the distance to the flattened boundary.}

\begin{lemma}[Derivatives of asymptotic profile parallel to flattened boundary]
\label{L14}
For any $\hat x \in \phi(\Omega\cap B_{{2}r}(x_0))$ and
$k \in \N$,  both
\[
\hat x_n \lesssim \hat V(\hat x) \lesssim \hat x_n
\]
and
\begin{equation}\label{linear tangent derivative growth}
|D^k_{\hat x'} \hat V(\hat x) | \lesssim \hat x_n,
\end{equation}
hold. {Moreover, $\hat V/\hat x_n$ belongs to $C^{2,\alpha}$ for some $\alpha \in(0,1)$ with $\alpha \le p-1$.}
\end{lemma}

\begin{proof}We start by noticing that the boundary estimate \eqref{24} translates into 
\begin{equation}
\label{39}
\hat x_n\lesssim \hat V(\hat x)\lesssim \hat x_n
\end{equation}
under the change of variables. Indeed, since 
\[
\dist(x,\p \Omega)^2 = \inf_{y'\in\R^{n-1}} \left(|x'-y'|^2 + (x_n-\gamma(y'))^2\right),
\]
on the one hand, by choosing $y'=x'$, we immediately deduce that
\[
\dist(x,\p \Omega) \le x_n-\gamma(x') = \hat x_n.
\]
On the other hand, as the minimizer $y'$ solves the optimality condition $
x'-y' = (x_n-\gamma(y'))\grad'\gamma(y')$, we find
\begin{align*}
\hat x_n  \le |x_n-\gamma(y')| + \|\grad'\gamma\|_{L^{\infty}} |x'-y'| \le \left(1+\|\grad'\gamma\|_{L^{\infty}}\right)\dist(x,\partial\Omega).
\end{align*}
Thus $\dist(x,\partial\Omega) $ is comparable to $\hat x_n$, which implies \eqref{39} via \eqref{24}.

We have to show that this estimate remains true for tangential derivatives. 
Since Lemma \ref{L13} asserts $\hat V \in C^{3,{\alpha}}$, \eqref{linear tangent derivative growth} follows directly for $k\in\{0,1,2\}$ via Taylor expansion because the homogeneous boundary conditions are invariant under differentiation in tangential direction. For larger values of $k$, we have to transform the elliptic equation \eqref{eq-V} into a problem on the half-space. In a similar way as we transformed the parabolic equation, we find that
\[
-\hat \grad\cdot(A\hat \grad(\hat \eta \hat V)) = \hat \eta\hat V^p + B\cdot \grad\hat V + C\hat V,
\]
for some smooth and bounded  functions $B$ and $C$ on $\R^n_+$ that depend only on the regularity and the shape of the boundary $\partial \Omega$. Differentiating with respect to tangential variables $x_i$ for any $i<n$, we find 
\[
-\hat \grad\cdot(A\hat \grad(\hat \eta \p_{\hat x_i}\hat V)) = f
\]
with $f \in C^{1,{\alpha}}$,  since e.g.~Lemma \ref{L13} shows $\p_{\hat x_i}( \hat V^p) = p \hat V^p \frac{\p_{\hat x_i} \hat V}{\hat V}$ 
to be the product of a $C^{{3},{\alpha}}$
function with a ratio in which the $C^{2,\alpha}$ numerator and $C^{3,\alpha}$ denominator both vanish linearly at the halfspace boundary.  
On any smooth bounded subdomain of the halfspace containing $\phi(\Omega \cap B_{2r(x_0)})$, Schauder theory {(e.g., Chapter 6 in \cite{GilbargTrudinger01}),} then implies $\hat \eta \p_{\hat x_i}\hat V \in C^{3,\alpha}$ so that
\eqref{linear tangent derivative growth} holds for $k=3$.  For larger $k$,  choosing a sequence $\hat \eta_{k-1} > \hat \eta_{k}$ of nested cutoffs satisfying the same hypotheses as $\hat \eta_3=\hat \eta$, and a multi-index $\beta \in \N_0^{n-1} \times \{0\}$ consisting of $|\beta|=k-2$ tangential derivatives yields
\[
-\hat \grad\cdot(A\hat \grad(\hat \eta_k \hat D^\beta\hat V)) = f_\beta.
\]
Induction on $k$ gives $f_\beta \in C^{1,{\alpha}}$ hence $\hat \eta_{k} \hat D^\beta \hat V \in C^{3,{\alpha}}$ and \eqref{linear tangent derivative growth} for all $k$;
this induction relies on the decay already established for the derivatives of $V$ which appear in the $p$-homogeneous nonlinearities,  (and the fact that 
of the $k-1$ derivatives of $V$ that contribute elsewhere to $f_\beta$, all but two are in tangential directions).

{Finally, the third statement of the lemma follows from Lemma \ref{L13} and Taylor expansion.}
\end{proof}
{It follows immediately from the preceding lemma that the  problem in \eqref{10e} can be further rewritten as
\begin{equation}
\label{10f}
\partial_t \hat H - \hat x_n^{-1-p} \hat \grad\cdot (\hat x_n^2 \tilde A \hat \grad \hat H) = \tilde F + \tilde G,
\end{equation}
for some new elliptic $\tilde A$, and where  $\tilde G$ is the sum of 
$\hat G$ and other lower-order terms of the same class. The weak formulation in \eqref{10} now turns into 
\begin{equation}\label{10h}
\begin{aligned}
\MoveEqLeft-\iint_{(0,T)\times \R^n_+} \hat H \partial_t \hat \varphi\, d\hat \mu_{p+1}dt + \iint_{(0,T)\times \R^n_+} \hat \grad \hat \varphi\cdot \tilde A \hat \grad  \hat H\, d \hat \mu_2dt\\
 &= \iint_{(0,T)\times \R^n_+}  \hat \varphi  (\hat F+\tilde G)\, d \hat \mu_{p+1}dt  + \int_{\R^n_+} \hat H_0 \left.\hat \varphi\right|_{t=0}\, d\hat \mu_{p+1},
\end{aligned}
\end{equation} 
for any {$\hat \varphi \in L^2(L^2_2)\cap L^2(\dot H^1_2)\cap \dot H^1(L^2_{p+1})$ vanishing near the  endpoint $T$}.

{We now prove maximal regularity for the problem in \eqref{10f}.
\begin{lemma}[Maximal regularity for linearized inhomogeneous halfspace problem]
\label{L14b}
Let $\tilde F$ and $\tilde G$ be given in $L^2(L^2_{2p})$ and let $\hat H\in L^2(L^2_{p+1})\cap L^2(\dot H^1_2)$ be a weak solution of \eqref{10f} with zero initial data. Then $\hat H \in L^2(\dot H^1)\cap L^2(\dot H^2_2)\cap \dot H^1(L^2_{2p})\cap C^0(L^2_{p+1})$ with 
\begin{equation}\label{10g}
\begin{aligned}
\MoveEqLeft \| \hat \grad\hat H\|_{L^{\infty}(L^2_{p+1})} + \|\partial_t \hat H\|_{L^2(L^2_{2p})} + \|\hat \grad \hat H\|_{L^2(L^2)} + \|\grad^2\hat H\|_{L^2(L^2_2)}\\
& \lesssim \|\tilde F\|_{L^2(L^2_{2p})} + \|\tilde G\|_{L^2(L^2_{2p})} + \|\hat \grad \hat H\|_{L^2(L^2_2)} +\|\hat H\|_{L^2(L^2_{2p})} .
\end{aligned}
\end{equation}
\end{lemma}
}

\begin{proof}

In order to simplify the notation, we drop the hats and tildes from here on. Moreover, we set $G=0$. We will here only give formal arguments. The estimates can be derived rigorously by approximating $F$ smoothly and  using suitable  finite difference quotient approximations for the test functions, see, e.g., Chapter 6.3 in \cite{Evans}.

In a first step, we use (an approximation with smooth cut-offs in time of) $\varphi  = -\chi_{(0,T)}\partial_k^2 H$ for some $k\in \{1,\dots,n-1\}$ as a test function in \eqref{10h}, where $\chi_{(0,T)}$ is the characteristic function for the time interval, and we obtain
\begin{align*}
\MoveEqLeft \iint_{(0,\infty)\times \R^n_+} H\partial_k^2\partial_t (\chi_{(0,T)}H)\, d\mu_{p+1}dt - \iint_{(0,T)\times \R^n_+} \grad (\partial_k^2 H)\cdot A \grad H\, d\mu_2 dt\\
& = - \iint_{(0,T)\times \R^n_+} \partial_k^2 H F\, d\mu_{p+1}dt.
\end{align*}
Multiple integrations by parts then yield
\begin{equation}\label{1000}
\begin{aligned}
\MoveEqLeft \frac12 \int_0^T \frac{d}{dt}\int_{\R^n_+} (\partial_k H)^2\, d\mu_{p+1}dt + \iint_{(0,T)\times \R^n_+} \grad \partial_k H \cdot A\grad \partial_k H\, d\mu_2 dt \\
& =  -\iint_{(0,T)\times \R^n_+} \partial_k^2 H F\, d\mu_{p+1}dt  - \iint_{(0,T)\times \R^n_+} \grad \partial_k H\cdot  (\partial_k A)\grad H\, d\mu_2dt,
\end{aligned}
\end{equation}
and invoking the ellipticity of the matrix $A$, the Cauchy--Schwarz inequality and recalling that $H$ was assumed to have zero initial data yields
\begin{align*}
\|\partial_k H\|_{L^{\infty}(L^2_{p+1})} + \|\grad\partial_k H\|_{L^2(L^2_2)} \lesssim \|\partial_k^2 H\|_{L^2{(L^2_2)}}^{\frac12} \|F\|_{L^2(L^2_{2p})}^{\frac12}+ \|\grad H\|_{L^2(L^2_2)}.
\end{align*}
Via the elementary estimate $ab \lesssim \eps a^2 + \eps^{-1}b^2$, we can control the second order term on the right-hand side. We have thus derived the desired control over the second order tangential and mixed derivatives, namely
\[
\|\partial_k H\|_{L^{\infty}(L^2_{p+1})} + \|\grad\partial_k H\|_{L^2(L^2_2)} \lesssim \|F\|_{L^2(L^2_{2p})}  + \|\grad H\|_{L^2(L^2_2)}.
\]
 Moreover, an application of Lemma \ref{L13b} provides also the control over the first order tangential derivatives, because $p+1>2$ implies:
  $$\|x_n^{-p} \partial_k H\|_{L^2_{2p}}  = \|\partial_k H\|_{L^2} \lesssim \|\partial_k H\|_{L^2_{p+1}} + \|\grad \partial_k H\|_{L^2_2}.$$
    Notice that a replacing the time interval $(0,T)$ in \eqref{1000} by $(t,t+\eps)$  shows also the continuity of $\|\partial_k H\|_{L^2_{p+1}}$ in time.

In order to control the transversal derivatives, it is now enough to focus on the $x_n$ variable, and thus study the one-dimensional problem
\[
\partial_t H - x_n^{-1-p} \partial_n (x_n^2 A_{nn}\partial_n H) = F,
\]
because
 all the other terms that appear in \eqref{10f} are now {known} to belong to $L^2_{2p}$. In order to simplify the notation further, we drop the subscripted $n$'s in the rest of the proof. Furthermore, the problem becomes more accessible if we freeze the diffusivity function $A$ at an arbitrary point $x_*$. That is, we study the equation
\[
\partial_t H - A_* x^{-1-p}\partial_x (x^2 \partial_x H)  = F - x^{-1-p}\partial_x (x^2 (A_*-A)\partial_x H),
\]
where we have set $A_*=A(x_*)$. Thanks to the regularity of $A$, it holds that $|A_*-A(x)|\lesssim \delta $ for $x$ and $x_*$ in the interval $ (0,\delta  )$. We should thus localize the problem further by smuggling in a cut-off function $\eta$ satisfying $\eta=1$ in $(0,\delta )$ and vanishing outside of $(0,2\delta )$. This way, we are led to considering
\begin{align*}
\MoveEqLeft \partial_t (\eta H) - A_* x^{-1-p}\partial_x (x^2 \partial_x (\eta H))\\
&  = \eta F + (A-A_*) x^{-1-p} \partial_x (x^2 \partial_x( \eta H)) -2x^{1-p}\partial_x \eta A\partial_x H\\
&\quad -2 x^{-p}\partial_x \eta AH + x^{1-p}\eta \partial_x A\partial_xH - x^{1-p}\partial_x^2\eta A H,
\end{align*}
and we write $\tilde F$ for the right-hand side for brevity and set $\tilde H  = \eta H$. Now, if we can show that 
\begin{equation}
\label{10i}
\|\partial_t \tilde H\|_{L^2_{2p}} +\|\partial_x \tilde H\|_{L^2} + \|\partial_x^2 \tilde H\|_{L^2_2} \lesssim \|\tilde F\|_{L^2_{2p}},
\end{equation}
the statement follows if $\delta$ is sufficiently small, because 
\[
\|\tilde F\|_{L^2_{2p}} \lesssim \|F\|_{L^2_{2p}} +   \delta \|\partial_x \tilde H\|_{L^2} + \delta  \|\partial_x^2\tilde H\|_{L^2_2} + \frac1{\delta} \|\partial_x H\|_{L^2_2} +\frac1{\delta} \|H\|_{L^2},
\] 
by the regularity of $A$ and the properties of the cut-off function. In view of the interpolation Lemma \ref{L13b}, the $L^2$ norm on $H$ can be replaced by the {$L^2_{2p}$} as in the statement of the lemma.

We have now reduced the multi-dimensional problem with variable coefficients \eqref{10f} to a one-dimension problem with constant coefficient,
\[
\partial_t \tilde H - A_* x^{-1-p}\partial_x (x^2 \partial_x \tilde H) =\tilde F.
\]
We have to do one more transformation in order to arrive at a problem that is better behaved. Indeed, if we change variables $\check x= x^{p+1}$, $\check t = A_* (p+1)^2t$,  $\check{H}(\check t,\check x) = \tilde H(t,x)$ and  $\check{ F}(\check t,\check x) = \tilde F(t,x)$ the above equation turns into
\[
\partial_{\check t} \check H - \check x^{-\sigma} \partial_{\check x}(\check x^{\sigma+1}\partial_{\check x} \check H) = \check F,
\]
where $\sigma = \frac1{p+1}$.
This is precisely the linear version of the parabolic equation that characterizes the porous medium dynamics in a neighborhood of the Barenblatt solution as studied earlier in \cite{Koch99,Kienzler16,Seis15+}. It is well-understood: Calder\'on--Zygmund and Muckenhoupt theory is available and provides estimates
\[
\|\partial_{\check t}\check H\|_{L^2_q} + \|\partial_{\check x}\check H\|_{L^2_{q}} + \|\partial_{\check x}^2 \check H\|_{L^2_{q+2}} \lesssim \|\check F\|_{L^2_q},
\]
for any $q\in(-1,2(\sigma+1)-1)$, see the proofs of Proposition 3.23 in \cite{Kienzler16} or Proposition 4.23 in \cite{Seis15+}. Our choice is $q = \frac{p}{p+1}$, which is equivalent to \eqref{10i}.

\end{proof}

We summarize our findings as follows:
\begin{proposition}[{Linear inhomogenous a priori estimates}]\label{P5}
Let $h$   be a {weak} solution to the linear equation \eqref{10} with zero initial datum and $f\in L^2(L^2_{2p}){\cap L^1(L^2_{p+1}})$. Then for all $T>0$, 
the following holds:
\begin{equation}
\label{15b}
\begin{aligned}
\MoveEqLeft \|\grad h\|_{L^{\infty}((0,T);L^2_{p+1})}+ \|\partial_t h\|_{L^2((0,T);L^2_{2p})} + \|\grad h\|_{L^2((0,T);L^2)} + \|\grad^2 h\|_{L^2((0,T);L^2_{2})}\\
&\lesssim  \|f\|_{L^2((0,T);L^2_{2p})} + {\|h\|_{L^2((0,T);L^2_{2p})}}.
\end{aligned}
\end{equation}
\end{proposition}

\begin{proof}
Since $\Omega \subset \R^n$ is bounded,  its boundary can be covered by finitely many open balls,  sufficiently small that
 within each of them,  the boundary can be expressed as a graph over any of its tangent planes.  The complement of these open sets in $\Omega$ can be covered by one additional open set compactly contained in the interior of $\Omega$. Choosing a partition of unity subordinate to this covering,  we flatten the boundary in each of the covering balls and apply Lemma \ref{L14b}.  The analogous estimates in the interior of $\Omega$ follow from standard parabolic estimates and the boundedness of $\log V$. Combining these estimates in the original variables using the partition of unity, the proposition follows from a linear analog of Lemma \ref{L17}.
\end{proof}

This maximal regularity result can be easily combined with the energy estimate for the nonlinear problem.  

\begin{lemma}[Nonlinear smoothing 1]
\label{L18}
Let $h $  be a solution to the nonlinear equation \eqref{eq-relativeerror} with  initial datum $h_0\in L^2_{p+1}$ and assume that $\|h\|_{L^{\infty}}\le \eps$ for some $\eps>0$ small enough.  Then, for any $0<\tau<1<T$ there exists $C=C(\tau,T,n,p,V)$ such that
\begin{equation}
\label{15}
 \|\partial_t h\|_{L^2((\tau,T);L^2_{2p}))} + \|\grad h\|_{L^2((\tau,T);L^2(L^2))} + \|\grad^2 h\|_{L^2((\tau,T);L^2_{2}))} \le  C \|h_0\|_{L^2_{p+1}}.
\end{equation}
\end{lemma}

\begin{proof}
We denote by $\zeta$ a smooth cut-off function that is $1$ in the time interval $[\tau, T]$ and zero in $[0,\tau/2]$. We localize the evolution {\eqref{eq-relativeerror}} with the help of this function
\begin{equation}\label{short time cutoff}
\partial_t(\zeta h)   - V^{-1-p}\div(V^2 \grad(\zeta h)) = (p-1)\zeta h   +  h \partial_t \zeta + \zeta N(h),
\end{equation}
and apply the maximal regularity estimate from Proposition \ref{P5} to the effect that
\begin{align*}
\MoveEqLeft \|\partial_t(\zeta h)\|_{L^2(L^2_{2p})} + \|\grad(\zeta h)\|_{L^2(L^2)} + \|\grad^2 (\zeta h)\|_{L^2(L^2_2)}\\
& \lesssim  {\|\zeta h\|_{L^2(L^2_{2p})} + \| h \partial_t \zeta \|_{L^2(L^2_{2p})} +\|
\zeta N(h)\|_{L^2(L^2_{2p})}}.
\end{align*}
Thanks to the particular structure of the nonlinearity {\eqref{eq-problemrelative},}  {the third term on} the right hand side can be estimated by
\begin{align*}
{\|
\zeta N(h)\|_{L^2(L^2_{2p})} \lesssim 
\|h \partial_t \zeta\|_{L^2(L^2_{2p})} + 
\|\zeta h\|_{L^2(L^2_{2p})}  + \eps \|\partial_t(\zeta h)\|_{L^2(L^2_{2p})}.}
\end{align*}
Using the fact that $L^2_{p+1}$ embeds continuously into $L^2_{2p}$ by the virtue of \eqref{24}, the above estimates combine with the pointwise bounds from Lemma \ref{L17} to give 
\begin{align*}
\MoveEqLeft \|\partial_t(\zeta h)\|_{L^2(L^2_{2p})} + \|\grad(\zeta h)\|_{L^2(L^2)} + \|\grad^2 (\zeta h)\|_{L^2(L^2_2)} \\
&\le C\left(\eps \|\partial_t(\zeta h)\|_{L^2(L^2_{2p})} + \|h_0\|_{L^2_{p+1}}\right).
\end{align*}
Choosing $\eps$ small enough  and invoking the properties of the cut-off function yields the statement of the lemma.
\end{proof}

Before turning to higher-order  derivatives, we {use integration by parts to establish a class of interpolation inequalities which will allow us to control the effects of the nonlinearity.}

\begin{lemma}[Interpolation]
\label{L16}
Let $\psi\in C^{\infty}_c(\R^n)$ be given, $q\ge2$ and $k,\ell\in \N$ with $k>\ell$. Then it holds
\[
\|\psi |D^{\ell} h|^{\frac{k}{\ell}}\|_{L^q} \lesssim \|h\|_{L^{\infty}}^{\frac{k-\ell}{\ell}} \sum_{m=0}^{k-1}  \|D^m\psi D^{k-m}h\|_{L^q}.
\]
\end{lemma}

\begin{proof}
To keep the notation as simple as possible, we perform a rather symbolic calculation. That is, we write $\partial^m$ for some partial derivative of $m$th order, $\partial^m = \partial_x^{\alpha}$ with $|\alpha|=m$. In our argument, the precise value of $\alpha$ is not of importance. 

We start with an integration by parts to notice that
\begin{align*}
&\int |\psi|^q |\partial^{\ell} h|^{\frac{kq}{\ell}}\, dx 
\\& = \int |\psi|^q \partial^{\ell} h \partial^{\ell} h |\partial^{\ell} h|^{\frac{kq-2\ell}{\ell}}\, dx\\
& \lesssim \int |\psi|^q |\partial^{\ell-1}h||\partial^{\ell+1}h||\partial^{\ell}h|^{\frac{kq-2\ell}{\ell }}\, dx
+ \int |\psi|^{q-1} |\partial \psi| |\partial^{\ell-1}h| |\partial^{\ell}h|^{\frac{kq-\ell}{\ell}}\, dx 
 \\& \lesssim  \left(\int |\psi|^q |\partial^{\ell-1} h|^{\frac{kq}{\ell-1}}\, dx\right)^{\frac{\ell-1}{kq}}\left(\int |\psi|^q |\partial^{\ell+1} h|^{\frac{kq}{\ell+1}}\, dx\right)^{\frac{\ell+1}{kq}}
 \left(\int |\psi|^q |\partial^{\ell}h|^{\frac{kq}{\ell}}\, dx\right)^{\frac{kq-2\ell}{kq}}\\
 &+ \left(\int |\psi|^q |\partial^{\ell-1} h|^{\frac{kq}{\ell-1}}\, dx\right)^{\frac{\ell-1}{kq}}\left(\int |\psi|^q |\partial^{\ell} h|^{\frac{kq}{\ell}}\, dx\right)^{\frac{kq-k-\ell+1}{kq}}
 \left(\int |\partial\psi|^q |\partial^{\ell}h|^{\frac{(k-1)q}{\ell}}\, dx\right)^{\frac{1}{q}},
\end{align*}
where the last estimates follow from H\"older's inequality. Here, we employ the convention that
\[
\left(\int |\psi|^q |\partial^{\ell-1} h|^{\frac{kq}{\ell-1}}\, dx\right)^{\frac{\ell-1}{kq}} = \|h\|_{L^{\infty}}\quad \mbox{for }\ell=1.
\]
Applying Young's inequality in the form $ab \lesssim \eps a^{\frac1{\theta}} + C_{\eps,\theta} b^{\frac1{1-\theta}}$ for some arbitrarily small $\eps$ and $\theta\in (0,1)$ to the right hand side yields
\begin{equation}\begin{aligned}\label{eq-youngs}
\int |\psi|^q  |\partial^{\ell} h|^{\frac{kq}{\ell}}\,dx & \lesssim \left(\int |\psi|^q |\partial^{\ell-1} h|^{\frac{kq}{\ell-1}}\, dx\right)^{\frac{\ell-1}{2\ell}}\left(\int |\psi|^q |\partial^{\ell+1}h|^{\frac{kq}{\ell+1}}\, dx\right)^{\frac{\ell+1}{2\ell}}\\
&\quad + \left(\int |\psi|^q |\partial^{\ell-1} h|^{\frac{kq}{\ell-1}}\, dx\right)^{\frac{\ell-1}{k+\ell-1}}\left(\int |\partial \psi|^q |\partial^{\ell}h|^{\frac{(k-1)q}{\ell}}\, dx\right)^{\frac{k}{k+\ell-1}}.\end{aligned}
\end{equation} 

It remains to apply an iteration procedure. For this purpose, we set for $k\ge \ell+m$,
\begin{align*}
A(k,\ell,m) &:= \left(\int |D^m\psi|^q |D^{\ell}h |^{\frac{(k-m)q}{\ell}}\, dx\right)^{\frac{\ell}{(k-m)q}},\\
B(\ell,m) & : = A(\ell+m,\ell,m),\\
C(k,\ell) &: = A(k,\ell,0)
\end{align*}
and notice that $
A_0:= A(k,0,m) = \|h\|_{L^{\infty}}$. Upon rescaling $h$, we may assume from here on that $A_0=1$. (We may always assume $A_0 \neq 0$ as otherwise the lemma is vacuously true.) 
With this notation, the previous estimate becomes
\begin{equation}
\label{29}
C(k,\ell) \lesssim C(k,\ell-1)^{\frac12}C(k,\ell+1)^{\frac12} + C(k,\ell-1)^{\frac{\ell}{k+\ell-1}} A(k,\ell,1)^{\frac{k-1}{k+\ell-1}},
\end{equation}
and the statement of the lemma can be rephrased as 
\begin{equation}
\label{28}
C(k,\ell)^{\frac k \ell}\lesssim \sum_{m=0}^{k-1} B(k-m,m),
\end{equation} for every $k$ and $1\le \ell \le k-1$. 

The proof will be a double-induction on $(k,\ell)$. We start by noticing that for $k=2$ and $\ell=1$, our objective \eqref{28} is nothing but the estimate \eqref{29} just proven.  Suppose $k\ge 3$ is fixed and that \eqref{28} holds true for  $k-1$ and any $\ell\le k-2$.  Our goal is to show \eqref{28} for fixed $k$ and all $1\le \ell \le k-1$.

We first need some auxiliary inequalities. Note for $\varphi = D\psi$, it holds that
\begin{align}
A(k,\ell,m)  
&= \left(\int |D^{m-1}\varphi|^q |D^{\ell}h|^{\frac{((k-1)- (m-1))q}{\ell}} \, dx\right)^{\frac{\ell}{((k-1)- (m-1))q}} 
\\&=: \tilde A(k-1,\ell,m-1)
\end{align}
and since the estimate in \eqref{29} is independent of the choice  $\psi$, the inductive hypothesis \eqref{28} allows us to estimate
\begin{align*}
 \tilde  A(k-1,\ell,0)^{\frac{k-1}{\ell}}
  =:  \tilde C(k-1,\ell)^{\frac{k-1}{\ell}}
& \lesssim      \sum_{m=0}^{k-1-1} \tilde B(k-1-m,m) 
\\& := \sum_{m=0}^{k-1-1} \tilde  A(k-1,k-1-m,m)
\end{align*}
or
\[
A(k,\ell,1))^{\frac{k-1}{\ell }} \lesssim \sum_{m=1}^{k-1} 
B(k-m,m),
\]
for any $\ell\le k-2$. Plugging this bound into \eqref{29} gives
\begin{equation}
\label{30}
\begin{aligned}
C(k,\ell) & \lesssim C(k,\ell-1)^{\frac12} C(k,\ell+1)^{\frac12} \\
&\quad + C(k,\ell-1)^{\frac{\ell}{k+\ell-1}} 
\left(\black{\sum_{m=1}^{k-1} B(k-m,m)}\right)^{\frac{\ell}{k+\ell-1}}.
\end{aligned}
\end{equation}

We claim that this estimate implies 
\begin{equation}\label{30a}
C(k,\ell)^{\frac{k}{\ell}} \lesssim C(k,\ell+1)^{\frac{k}{\ell+1}} + \sum_{m=1}^{k-1} B(k-m,m)
\end{equation}
for any $1\le \ell\le k-1$. Indeed, the case $\ell =1$ follows directly from \eqref{30} because $C(k,0)=A_0=1$.  The general case follows by induction: We suppose that \eqref{30a} is proved for $1,2,\dots,\ell-1$ and we aim at establishing it for $\ell$. For this purpose, we use Young's inequality  in \eqref{30} to the effect that
\[
C(k,\ell)^{\frac{k}{\ell}} \lesssim \eps C(k,\ell-1)^{\frac{k}{\ell-1}} + C(k,\ell+1)^{\frac{k}{\ell+1}} + \sum_{m=1}^{k-1}B(k-m,m),
\]
for some arbitrary $\eps$. Invoking the hypothesis that \eqref{30a} holds true for $\ell-1$, we then deduce
\[
C(k,\ell)^{\frac{k}{\ell}} \lesssim \eps C(k,\ell)^{\frac{k}{\ell}} + C(k,\ell+1)^{\frac{k}{\ell+1}} + \sum_{m=1}^{k-1}B(k-m,m),
\]
which gives \eqref{30a} for $\ell$ if $\eps$ is chosen small enough.

It remains to iterate \eqref{30a} to find
\[
C(k,\ell)^{\frac{k}{\ell}} \lesssim C(k,k) + \sum_{m=1}^{k-1} B(k-m,m) = \sum_{m=0}^{k-1} B(k-m,m),
\]
which is what we aimed to prove, cf.~\eqref{28}.
\end{proof}

We will now perform an intermediate  step towards higher-order regularity estimates by lifting the norms on the left-hand side in \eqref{15} to the next order in time. Higher-order time derivatives will be considered subsequently simultaneously with suitable higher-order spatial derivatives.  The intermediate step that we take in the following lemma  is necessary in order to control lower-order error terms that appear later as a result of a  transformation of the equation close to the boundary, see \eqref{21} below. 

\begin{lemma}[Nonlinear smoothing 2]
\label{L19}
Let $h$ be a solution to the equation \eqref{eq-relativeerror} with initial datum $h_0\in L^2_{p+1}$ and assume that $\|h\|_{L^{\infty}}\le \eps$ for some $\eps>0$. Then if $\eps$ is small enough and $0<\tau<1<T$, it holds that
\begin{equation}
\label{33}
 \|\partial_t^2 h\|_{L^2((\tau,T);L^2_{2p}))} + \|\grad\partial_t  h\|_{L^2((\tau,T);L^2(L^2))} + \|\grad^2\partial_t h\|_{L^2((\tau,T);L^2_{2}))}\lesssim  \|h_0\|_{L^2_{p+1}}.
\end{equation}
\end{lemma}

\begin{proof}
{Regularity in time was proved already by Jin and Xiong, see \eqref{200} above. In order to get control over the mixed derivatives, we proceed carefully by considering finite difference quotients $d_t^s h (t) = s^{-1}(h(t+s)-h(s))$. 
We consider the same cut-off function in time as in the proof of Lemma \ref{L18}. Then localizing the nonlinear equation and ``differentiating'' \eqref{short time cutoff} 
with respect to time, we obtain
\begin{align*}
\MoveEqLeft \partial_t d_t^s (\zeta h) -V^{1-p}\div(V^2\grad d_t^s(\zeta h))  \\
&= (p-1)d_t^s (\zeta h) + d_t^s h \partial_t \zeta + hd_t^s \partial_t  \zeta  +\zeta d_t^s N(h) +  N(h) d_t^s \zeta.
\end{align*}
}
Regarding the nonlinear terms, we notice that
\[
|N(h)| \lesssim |h|^2 + |h||\partial_t h|
\]
and 
\[
|d_t^s {N}(h)| \lesssim |h||d_t^s h| + |d_t^sh||\partial_t h| + |h||\partial_t d_t^s h|.
\]
Therefore, using $|h|\le \eps\le 1$ and making use of the maximal regularity estimate for the linear problem, Proposition \ref{P5}, we find that 
\begin{align*}
\MoveEqLeft\|\partial_t d_t^s(\zeta h)\|_{L^2(L^2_{2p})} + \|\grad d_t^s(\zeta h)\|_{L^2(L^2)} + \|\grad^2 d_t^s (\zeta h)\|_{L^2(L^2_2)}\\
 &\lesssim \|\chi_{\spt\zeta}  h\|_{L^2(L^2_{2p})} + \|\chi_{\spt\zeta} d_t^s h\|_{L^2(L^2_{2p})}\\
 &\quad  + \|\zeta d_t^sh\partial_t h\|_{L^2(L^2_{2p})} + \eps \|d_t^s\partial_t(\zeta h)\|_{L^2(L^2_{2p})} .
\end{align*}
If $\eps$ is sufficiently small, the last term on the right-hand side can be absorbed into the left-hand side. Moreover, the (remaining) expressions on the right-hand side are bounded uniformly in $s$ by the virtue of Jin and Xiong's regularity statement \eqref{200}. We may thus pass to the limit $s\to 0$ and find 
\begin{align*}
\MoveEqLeft\|\partial_t^2(\zeta h)\|_{L^2(L^2_{2p})} + \|\grad \partial_t (\zeta h)\|_{L^2(L^2)} + \|\grad^2 \partial_t (\zeta h)\|_{L^2(L^2_2)}\\
 &\lesssim \|\chi_{\spt\zeta}  h\|_{L^2(L^2_{2p})} + \|\chi_{\spt\zeta} \partial_t h\|_{L^2(L^2_{2p})}+ \|\zeta  (\partial_t h)^2\|_{L^2(L^2_{2p})} .
\end{align*}
Now, applying the interpolation Lemma \ref{L16} in the form
\[
\|\zeta |\partial_t h|^2 \|_{L^2(L^2_{2p})} \lesssim \|h\|_{L^{\infty}} \|\zeta \partial_t^2 h\|_{L^2(L^2_{2p})} + \|h\|_{L^{\infty}}\|\partial_t\zeta\partial_t h\|_{L^2(L^2_{2p})},
\]
and using again that $|h|\le\eps\le 1$ leads us to the estimate
\begin{align*}
\MoveEqLeft
\|\partial_t^2(\zeta h)\|_{L^2(L^2_{2p})} + \|\grad \partial_t(\zeta h)\|_{L^2(L^2)} + \|\grad^2 \partial_t (\zeta h)\|_{L^2(L^2_2)}\\
 &\lesssim \|\chi_{\spt\zeta}  h\|_{L^2(L^2_{2p})} + \|\chi_{\spt\zeta} \partial_t h\|_{L^2(L^2_{2p})}   ,
\end{align*}
provided that $\eps$ is sufficiently small. Since $\mu_{2p}\lesssim \mu_{p+1}$ by the virtue of Lemma \ref{L13}, we can now apply Lemmas \ref{L17} and \ref{L18} and deduce the statement of the lemma by the properties of the cut-off function.
\end{proof}

Similarly to the derivation of the maximal regularity estimate in Proposition~\ref{P5}, the  derivation of higher-order regularity estimates requires attention only in a neighborhood of the boundary. Indeed, in the interior the equation is parabolic with smooth coefficients, and thus, higher-order estimates in the interior  just follow by standard iterative arguments based on the maximal regularity estimate from Proposition \ref{P5}. As before, we shall thus focus on the boundary from here on. We choose essentially the same notation as in the proof of Proposition \ref{P5} and we fix $x_0\in \partial\Omega$ arbitrarily and let $\eta$ denote a cut-off function on $\R^n$ interpolating smoothly between $\eta=1$ in $B_r(x_0)$ and $\eta = 0$ outside of $B_{2r}(x_0)$. 
  Moreover, as in the proofs of Lemmas \ref{L18} and \ref{L19}, we have to introduce a cut-off function $\zeta$ defined on $[0,\infty)$ that satisfies $\zeta=0$ in $[0,\tau/2]$ and $\zeta=1$ in $[\tau,T]$ for some $0<\tau<1<T$. Treating the nonlinearity $N(h(x))=:f(x)$ as an inhomogeneity  and smuggling $\zeta \eta$ into the equation, we find that $H = \zeta \eta h$ satisfies
\begin{equation}\label{21}
\partial_t H - V^{-1-p}\div(V^2\grad H) = (p-1)H +  F+G,
\end{equation}
where 
\[
F = \zeta \eta f,\quad G =  -2V^{1-p}\zeta \grad\eta\cdot \grad  h - V^{1-p}\zeta h \laplace \eta   - 2\zeta V^{-p} h \grad V \cdot \grad \eta  + \zeta' \eta h.
\]
We now apply the same diffeomorphism $\phi$ that we used in order to transform the elliptic problem \eqref{10d} into the half-space problem \eqref{10f}, and arrive at
\begin{equation}
\label{18}
\partial_t \hat H  - \hat V^{-1-p}\hat\grad\cdot(\hat V^2 A\hat\grad\hat H)   = (p-1)\hat H +\hat F + \hat G.
\end{equation}
We will now derive control on higher-order derivatives for equation \eqref{18}. As before, we interpret the weighted Lebesgue norms  with respect to the simpler weight $\hat x_n$ and we consider
$L^r(L^2_{q}) = L^r((0,T);L^2(\hat \mu_q))$  with measure $d\hat \mu_q  = \hat x_n^q\, d\hat x$ {and typically $r=2$}. Moreover, we write $z=(t,\hat x')$
{for the time and flattened tangential variables, whereas $\hat \nabla$ denotes the full spatial gradient (tangential and normal) in flattened coordinates}.

\begin{lemma}[Tangential smoothing by the linear inhomogeneous evolution]
\label{L10}
Let $\hat H$ be a solution to the transformed equation {\eqref{18}}. For any $k\in\N$ and $\alpha'\in\N_0^{n}$ {with $|\alpha'|=k$}, it holds that
\begin{equation}
\label{20}
\begin{aligned}
\MoveEqLeft \|D_z^k\partial_t  \hat H\|_{L^2(L^2_{2p})} + \|D_z^k  \hat \grad \hat H\|_{L^2(L^2)} + \|D_z^k   \hat \grad^2 \hat H\|_{L^2(L^2_{2})}\\
& \lesssim  \sum_{\ell=0}^k \|D_z^{\ell}\hat F\|_{L^2(L^2_{2p})}  +  \sum_{\ell=0}^k \|D_z^{\ell} \hat G\|_{L^2(L^2_{2p})}.
\end{aligned}
\end{equation}
\end{lemma}

\begin{proof} As $\hat G$ can be considered as an inhomogeneity, we can set $\hat G=0$ for notational convenience. 

We can proceed as in the proof of Proposition \ref{P5} and show that {\eqref{15b}} holds true on the half-space, that is, we have
\begin{equation}
\label{19}
 \|\partial_t \hat  H\|_{L^2(L^2_{2p})} + \|\hat  \grad \hat H\|_{L^2(L^2)} + \|\hat \grad^2 \hat H\|_{L^2(L^2_{2})}\lesssim  \|\hat F\|_{L^2(L^2_{2p})}.
\end{equation}
Note that the implicit constant in this estimate might be time-dependent and blow up for infinite times. Now, we differentiate \eqref{18} in time and tangential direction. For $m\in \N_0$ and $\alpha'\in \N_0^{n-1}$, it holds that
\begin{align*}
\MoveEqLeft
\partial_t \partial_t^m\partial_{\hat x'}^{\alpha'} \hat H  - \hat V^{-1-p}\hat \grad\cdot( V^2 A\hat \grad \partial_t^m\partial_{\hat x'}^{\alpha'} \hat H) \\
&  = (p-1)\partial_t^m\partial_{\hat x'}^{\alpha'} \hat H  +\partial_t^m\partial_{\hat x'}^{\alpha'} \hat F\\
&\qquad + \hat x_n^{1-p} \sum_{\substack{1\le |\beta|\le |\alpha'|+1 \\ \beta_n\le 2}} a_{\beta} \partial_t^m\partial_{\hat x}^{\beta} \hat H + \hat x_n^{-p} \sum_{\substack{1\le |\beta|\le |\alpha'|\\ \beta_n\le 1}} b_{\beta} \partial_t^m\partial_{\hat x}^{\beta}\hat  H ,
\end{align*}
for some {continuous and bounded functions $a_{\beta}, b_{\beta}$ on $\R$}.
 We can now apply the maximal regularity estimate \eqref{19} and find \eqref{20}
via iteration and thanks to the fact that $\hat x_n\lesssim 1$ in the support of $\hat H$. 
\end{proof}

Now, we translate the above estimate to the nonlinear setting. That is, we consider $f=f_0+f_1$, with 
\[
f_0  =   (1+h)^p - 1 - p h ,\quad f_1  = \left((1+ h)^{p-1}-1 \right) \partial_t h ,
\]
and we write $\hat F_i(t,\hat x) = \zeta(t) \hat \eta(\hat x) \hat f_i (\hat x)  = \zeta(t)\eta(x) f_i(x)$ for the transformed and truncated quantities. Moreover, we write $\hat h(t,\hat x) = h(t,x)$.  The nonlinearities are bounded as follows.

\begin{lemma}[{Spacetime localized boundary estimates for the nonlinearity}]
\label{L11}

Suppose that $\|\hat h\|_{L^{\infty}}\le\eps$ for some $\eps \ll 1$. Then, for any $k\in \N_{0}$ {there exists} a constant $\nu\in(0,1)$ such that
\begin{equation}\label{13b}
\begin{aligned}
\|D_z^k  \hat F_i\|_{L^2(L^2_{2p })} \lesssim \eps^{\nu} \sum_{m=0}^{k}  \|D_z^{k-m}(\zeta \hat \eta) D_z^{m+i} \hat h\|_{L^2(L^2_{2p})}  \qquad
{\forall i \in \{0,1\}}.
\end{aligned}
\end{equation}
\end{lemma}

\begin{proof}We drop the hats for notational convenience.
We start considering the estimate for $ F_0$, and notice that
\[
\partial_{z}^{\alpha} F_0  = \sum_{\beta\le \alpha} {\alpha\choose \beta}  \partial_{z}^{\alpha- \beta} \psi\,   \partial_{z}^{\beta} f_0,
\]
by the multi-dimensional Leibniz rule, where we have set $\psi = \zeta\eta$. We inspect the nonlinearity and find by Young's inequality and an iterative argument that
\[
| \partial^{\beta}_{z} f_0| \lesssim |h| | \partial_{z}^{\beta} h| + \sum_{\substack{\gamma\le \beta \\ 1\le |\gamma|\le |\beta|-1}} | \partial_{z}^{\gamma} h |^{\frac{|\beta|}{|\gamma|}},
\]
provided that $\eps$ is sufficiently small and $|\beta|\ge1$.
Therefore, summing over any multi-indices $\alpha$ with  $|\alpha|=k$ and integrating in space and time, we have that
\begin{align*}
 \| D_z^k F_0\|_{L^2(L^2_{2p })}
 &\lesssim \|h\|_{L^{\infty}} \|\psi D_z^k h\|_{L^2(L^2_{2p})}
+ \sum_{\ell=1}^{k-1}  \|\psi |D_z^{\ell} h |^{\frac{k}{\ell}}\|_{L^2(L^2_{2p})}
  \\&\qquad 
+ \sum_{{m=0}}^{k-1} \|h\|_{L^{\infty}} \|  D_z^{k-m} \psi  D_z^m h\|_{L^2(L^2_{2p})}
\\&\qquad 
+ \sum_{{m=2}}^{k-1}
\sum_{\ell=1}^{{m-1}} \|  D_z^{k-m}\psi | D_z^{\ell} h |^{\frac{m}{\ell}}\|_{L^2(L^2_{2p})}.
\end{align*} 
Let's discuss the right-hand side term by term. The first term is exactly of the kind we are looking for. For the second one, we apply Lemma \ref{L16} and find,
\[
\sum_{\ell=1}^{k-1} \|\psi |D_z^{\ell} h |^{\frac{k}{\ell}}\|_{L^2(L^2_{2p})} \lesssim \eps^{\nu}   \sum_{m=0}^{k-1} \|D_z^m\psi D_z^{k-m} h\|_{L^2(L^2_{2p})},
\]
for some $\nu>0$. The remaining terms are bounded by the same quantity.

The treatment of $f_1$ is similar. This time, derivatives of the nonlinearity are bounded as follows,
\begin{align*}
|D_z^{\ell} f_1| & \lesssim |h| |D_z^{\ell}\partial_t h| + \sum_{m=0}^{\ell-1} \sum_{n=1}^{\ell-m} |D_z^n h|^{\frac{\ell-m}{n}} |D_z^{m}\partial_t h|,
\end{align*}
as can be observed by Young's inequality and an iterative argument. With the help of the Leibniz rule, we thus obtain
\begin{align*}
\|D_z^k F_1\|_{L^2(L^2_{2p})} 
&\lesssim  \|h\|_{L^{\infty}} \|\psi D_z^k\partial_t h\|_{L^2(L^2_{2p})}
+\sum_{\ell=0}^{k-1} \|h\|_{L^{\infty}} \|D^{k-\ell}_z \psi D_z^{\ell} \partial_t h\|_{L^2(L^2_{2p})}
\\&\qquad + \sum_{m=0}^{k-1} \sum_{n=1}^{k-m} \|\psi |D_z^n h|^{\frac{k-m}{n}} D_z^{m} \partial_t h\|_{L^2(L^2_{2p})}
\\&\qquad +\sum_{\ell=0}^{k-1} \sum_{m=0}^{ \ell-1} \sum_{n=1}^{ \ell-m} \|D_z^{k-\ell}\psi |D^nh|^{\frac{\ell-m}{n}} D^m_z \partial_t h\|_{L^2(L^2_{2p})},
\end{align*} 
and absorbing $\partial_t$ into $D_z$ and an application of the H\"older inequality furthermore yields
\begin{align*}
\MoveEqLeft
\|D_z^k F_1\|_{L^2(L^2_{2p})}
\\ 
\lesssim& \|h\|_{L^{\infty}}\left( \|\psi D_z^{k+1}h\|_{L^2(L^2_{2p})}
+\sum_{\ell=1}^{k-1} 
 \|D^{k-\ell}_z \psi D_z^{\ell+1} h\|_{L^2(L^2_{2p}) }\right)
\\&
+ \sum_{m=0}^{k-1} \sum_{n=1}^{k-m} \|\psi |D_z^{n} h|^{\frac{k+1}{n}}\|_{L^2(L^2_{2p})}^{\frac{k-m}{k+1}}\|\psi |D_z^{m+1} h|^{\frac{k+1}{m+1}}\|_{L^2(L^2_{2p})}^{\frac{m+1}{k+1}}
\\&
+\sum_{\ell=0}^{k-1} \sum_{m=0}^{ \ell-1} \sum_{n=1}^{ \ell-m} \|D_z^{k-\ell}\psi |D_z^n h|^{\frac{\ell+1}n}\|_{L^2(L^2_{2p})}^{\frac{\ell-m}{\ell+1}} \|D_z^{k-\ell}\psi |D_z^{m+1}h|^{\frac{\ell+1}{m+1}}\|_{L^2(L^2_{2p})}^{\frac{m+1}{\ell+1}}.
\end{align*} 
We now invoke  Lemma \ref{L16} to the effect that
\[
\|D_z^{k} F_1\|_{L^2(L^2_{2p})} \lesssim \eps^{\nu} \sum_{m=0}^k \|D_z^m \psi D_z^{k+1-m} h\|_{L^2(L^2_{2p})},
\]
for some $\nu>0$. 
\end{proof}

With these preparations, we are in the position to extend the $L^2$ estimates from Lemmas \ref{L17}, \ref{L18}, and \ref{L19} to $z$-derivatives of any order {where $z=(t,\hat x')$ denotes the tangential and time variables and $\hat \nabla$ denotes the full spatial gradient (tangential and normal) in flattened coordinates}. For our purposes, it is enough to {bound the unweighted terms in these estimates.}

\begin{proposition}[Tangential nonlinear smoothing in Hilbert norms]
\label{P20}
Let $\hat H$ be the  solution to the transformed equation {\eqref{18}} and suppose that $\|\hat H\|_{L^{\infty}}\le \eps$ for some $\eps$ small enough. Let $0<\tau<1<T$ be given. Then, for any $k\in \N_0$, it holds that
\[
\|D_z^k \hat \grad\hat H\|_{L^2(L^2)} + \|D_z^k\hat \grad^2 \hat H\|_{L^2(L^2_2)}\lesssim \|h_0\|_{L^2_{p+1}}.
\]
\end{proposition}

\begin{proof}
We start by noting that the estimates from Lemmas \ref{L17}, \ref{L18} and \ref{L19} easily translate into the localized setting, so that
\begin{equation}
\label{34}
\begin{aligned}
\MoveEqLeft \|\hat H\|_{L^{\infty}((0,T); L^2_{p+1})} + \|\partial_t \hat H\|_{L^2((\tau,T);L^2_{2p})}+ \|\partial_t^2 \hat H\|_{L^2((\tau,T);L^2_{2p})}\\
\MoveEqLeft  \quad  + \|\hat \grad \hat H\|_{L^2((\tau,T);L^2)} + \|\hat \grad\partial_t \hat H\|_{L^2((\tau,T);L^2)} + \|\hat \grad^2 \hat H\|_{L^2((\tau,T);L^2_2)}\\
&\quad  + \|\hat \grad^2 \partial_t \hat H\|_{L^2((\tau,T);L^2_2)} 
\lesssim \|h_0\|_{L^2_{p+1}}.
\end{aligned}
\end{equation}

In order to derive estimates on derivatives of the next order, we invoke Lemma \ref{L10} with $k=1$,
\begin{equation}\label{35}
\begin{aligned}
\MoveEqLeft\|D_z\partial_t \hat H\|_{L^2(L^2_{2p})} + \|D_z\hat \grad \hat H\|_{L^2(L^2)} + \|D_z\hat \grad^2 \hat H\|_{L^2(L^2_2)} \\
& \lesssim \|\hat F\|_{L^2(L^2_{2p})} + \|D_z \hat F\|_{L^2(L^2_{2p})} + \|\hat G\|_{L^2(L^2_{2p})} + \|D_z\hat G\|_{L^2(L^2_{2p})}.
\end{aligned}
\end{equation}
Of course, our presentation here is a bit formal: Instead of considering derivatives $D_z$, we should more carefully apply difference quotients to the nonlinear equation. We have done so in Lemma \ref{L19} to deal with time derivatives. Because of the known reguarity  in time \eqref{200}, passing to the limit in the difference quotients don't cause any problems, also in the nonlinear terms. The strategy remains the same for higher order derivatives in time, and we may generously simplify our presentation here by considering proper time derivatives in the sequel.

When it comes to higher order derivatives in tangential direction, a result analogous to \eqref{200} is missing, but will be derived by us in Corollary \ref{C12} below. We should thus be a bit more careful in our argumentation. For the sake of a simpler presentation though, we shall keep the notation $D_z$ and will tacitly interpret it as  difference quotients in the tangential variables. Only in the discussion of the leading order nonlinear terms, we shall recall its actual meaning.  For all other terms we will be rather formal.

Let us start considering the lower-order terms. From the definition of $\hat G$, we deduce
\begin{align*}
 \|\hat G\|_{L^2(L^2_{2p})} + \|D_z \hat G\|_{L^2(L^2_{2p})} 
&\lesssim \|\chi_{\spt \psi} \hat  h\|_{L^2(L^2)} + \|\chi_{\spt \psi} \hat \grad h\|_{L^2(L^2_2)}\\&\quad + \|\chi_{\spt \psi}D_z \hat  h\|_{L^2(L^2)} + \|\chi_{\spt \psi} D_z \hat \grad \hat h\|_{L^2(L^2_2)},
\end{align*}
where we have set $\psi = \zeta \hat \eta$.

We will now apply an interpolation to {modify the weights} on the right-hand side. Let $\phi$ be a cut-off function that is $1$ in the support of $\psi$ and vanishes outside of a small neighborhood $\widehat{\spt \psi}$ of this support. Then it holds for any regular function $\xi$ that
\begin{align*}
\int \phi^2 \xi^2 \, d\hat x & = \int {\left(\frac{d\hat x_n}{d\hat x_n}\right)} \phi^2 \xi^2\, d\hat x\\
& = -2 \int \hat x_n \phi \xi \partial_n \xi\, d\hat x  - 2\int \hat x_n \phi \partial_n \phi \xi^2\, d\hat x.
\end{align*}
An application of the Cauchy--Schwarz inequalities then yields 
\[
\| \chi_{\spt \psi }\xi\|_{L^2(L^2)} \lesssim \|\chi_{\widehat{\spt \psi}} \xi\|_{L^2(L^2_2)} + \|\chi_{\widehat{\spt \psi}} \hat \grad \xi\|_{L^2(L^2_2)}.
\]
This argument can be repeated by writing $\hat x_n^q = \frac1{q+1} \frac{d}{d\hat x_n}\hat x_n^{q+1}$ and using that $\hat \mu_q\lesssim \hat \mu_2$ for $q\ge 2$, to derive
\begin{equation}\label{36}
\| \chi_{\spt \psi }\xi\|_{L^2(L^2)} \lesssim \|\chi_{\widehat{\spt \psi}} \xi\|_{L^2(L^2_q)} + \|\chi_{\widehat{\spt \psi}} \hat \grad \xi\|_{L^2(L^2_2)},
\end{equation} 
for any $q\in 2\N$. By interpolation, this estimates extends to  any $q>0$.

Making use of this interpolation-type estimate with suitable choices of $q$ in the above estimate for $\hat G$ and estimating $\hat \mu_{2p}\lesssim \hat\mu_{p+1}\lesssim \hat\mu_{2} \lesssim 1$ yields
\begin{align*}
 \|\hat G\|_{L^2(L^2_{2p})} + \|D_z \hat G\|_{L^2(L^2_{2p})} 
&\lesssim \|\chi_{\widehat{\spt \psi}} \hat  h\|_{L^2(L^2_{p+1})} + \|\chi_{\widehat{\spt \psi}} \hat \grad h\|_{L^2(L^2)}\\&\quad + \|\chi_{\widehat{\spt \psi}}\partial_t \hat  h\|_{L^2(L^2_{2p})} + \|\chi_{\widehat{\spt \psi}} \partial_t \hat \grad \hat h\|_{L^2(L^2)}\\
&\quad+ \|\chi_{\widehat{\spt \psi}} \hat \grad^2 \hat h\|_{L^2(L^2_2)},
\end{align*}
We may now invoke \eqref{34} with suitable choices of $\tau$, $T$ and {spatial cutoff} $r$ to deduce that 
\[
 \|\hat G\|_{L^2(L^2_{2p})} + \|D_z \hat G\|_{L^2(L^2_{2p})}  \lesssim \|h_0\|_{L^2_{p+1}}.
 \]
 
It remains to estimate the nonlinear terms in \eqref{35}. We promised to be more careful when considering higher order tangential difference quotients and we should thus briefly discuss the rigorous treatment of the leading order nonlinear terms. Because tangential derivatives leave the limit function $\hat V$ invariant in terms of scaling, cf.~Lemma \ref{L14}, these are terms of the order $ \psi |d_i^s \hat h||\partial_t \hat h|$ and $ \psi |\hat h||\partial_td_i^s \hat h|$, where $d_i^s$ is the difference quotient operator in direction $\hat x_i$, see also Lemma \ref{L19}. By using the smallness of $\hat H$ in the assumption, the second of these terms can be absorbed into the left-hand side before passing to the limit $s\to 0$. The other term can be split into the two quadratic terms $\psi (d_i^s \hat h)^2$ and $\psi (\partial_t \hat h)^2$, among which we only have to consider the first one, because time regularity is already settled. Here, we notice that this term is bounded uniformly in $s$, because $\|\psi (\partial_i \hat h)^2\|_{L^2_{2p}}$ can be estimated by $\|\hat h\|_{L^{\infty}} \|\psi \partial_i^2 \hat h\|_{L^2_{2}}$ plus lower order terms via the interpolation Lemma \ref{L16} and by using $\hat \mu_{2p}\lesssim \hat \mu_2$ in the support of $\hat \eta$.  This term is controlled via \eqref{34}.
There are no further regularity issues popping up when considering higher order tangential derivatives. We shall continue with the rather formal discussion and summarize here that
Lemma \ref{L11} implies
\begin{align*}
\MoveEqLeft \|\hat F\|_{L^2(L^2_{2p})} + \|D_z\hat F\|_{L^2(L^2_{2p})} \\
& \lesssim \eps^{\nu}\left( \|\psi \hat h\|_{L^2(L^2_{2p})} + {\|\psi D_z \hat h\|_{L^2(L^2_{2p})}} + {\|\psi D^2_z \hat h\|_{L^2(L^2_{2p})}} \right.\\
&\qquad  +\left. \|D_z \psi \hat h\|_{L^2(L^2_{2p})} 
+ \|D_z \psi {D_z}  \hat h\|_{L^2(L^2_{2p})}\right)
\\&\lesssim \|\chi_{\spt\psi} \hat h\|_{L^2(L^2_{2p})} +\|\chi_{\spt\psi} D_z \hat h\|_{L^2(L^2_{2p})} +\eps^{\nu} \| D^{2}_z  \hat H\|_{L^2(L^2_{2p})}
\\&\lesssim \|h_0\|_{L^2_{p+1}} + \eps^{\nu} \|D_z \p_t \hat H\|_{L^2(L^2_{2p})} + {\eps^\nu \|D_z \hat \nabla \hat H\|_{L^2(L^2)}},
\end{align*} where the last inequality is due to \eqref{34}.

Plugging these estimates {(for the lower-order terms involving $\hat G$ and the nonlinearities involving $\hat F$)} into \eqref{35} thus yields
\[
\|D_z\partial_t \hat H\|_{L^2(L^2_{2p})} + \|D_z\hat \grad \hat H\|_{L^2(L^2)} + \|D_z\hat \grad^2 \hat H\|_{L^2(L^2_2)} \lesssim \|h_0\|_{L^2_{p+1}},
\]
provided that $\eps^\nu$ is chosen small enough {that the final terms above which it multiplies can be absorbed into the left hand side}.

This procedure can be iterated, with a suitable adaption of {$\tau, T$ and the radii $r$} of the spatial cut-off functions $\eta$ in each step 
{to prove}\[
\|D_z^k\partial_t \hat H\|_{L^2(L^2_{2p})} + \|D_z^k\hat \grad \hat H\|_{L^2(L^2)} + \|D_z^k\hat \grad^2 \hat H\|_{L^2(L^2_2)} \lesssim \|h_0\|_{L^2_{p+1}},
\]
{inductively. This implies}  the desired bounds.
\end{proof}

Finally, we use generalized Sobolev embeddings to pass from $L^2$ to $L^\infty$ estimates: 
\begin{corollary}[Tangential nonlinear smoothing in {weighted} uniform norms]
\label{C12}
Let $\hat H$ be the  solution to the transformed equation {\eqref{18}} 
and suppose that $\|\hat H\|_{L^{\infty}}\le \eps$ for some $\eps$ small enough. Let $0<\tau<1<T$ be given. Then, for any $k\in \N_0$, it holds that
\[
\|D_z^k  \hat H\|_{L^{\infty}(L^{\infty})} + \|\hat x_n D_z^k \hat \grad \hat H\|_{L^{\infty}(L^{\infty})}\lesssim \|h_0\|_{L^2_{p+1}}.
\]
\end{corollary}
\begin{proof}The estimate basically follows from Proposition \ref{P20} via {generalized Sobolev embeddings using compact support in the $z=(t,x')\in \R^n$ variables,
followed by (two) integrations in $x_n$ where we have a vanishing boundary condition at one end $x_n=r$ only.} Indeed, for $m\in \N$ with $m>{n/2}$, it holds that
\begin{align*}
\|D_z^k\hat H\|_{L^{\infty}(L^{\infty})}
& \lesssim \sum_{\ell=0}^m \|D_z^{k+\ell}\hat H\|_{L^{2}(L^{2})} +\sum_{\ell=0}^m \|D_z^{k+\ell}\partial_{\hat x_n}\hat H\|_{L^{2}(L^{2})}.
\end{align*}
We now use the Hardy inequality
\[
\|\xi\|_{L^2(L^2)} \lesssim \|\partial_{\hat x_n}\xi\|_{L^2(L^2_2)},
\]
whose proof is similar to that of \eqref{36}, to eliminate the zero-order terms on the right-hand side, thus,
\[
\|D_z^k\hat H\|_{L^{\infty}(L^{\infty})} \lesssim \sum_{\ell=0}^{m+1} \|D_z^{k+\ell}\hat\grad\hat H\|_{L^{2}(L^{2})}.
\]
{Finally, we} apply Proposition \ref{P20} to infer the desired control of the first term in the  statement of the lemma. The  second term is bounded analogously, {by applying the same
argument to $\hat x_n D^k \hat \nabla \hat H$ in place of $D^k \hat H$}.
\end{proof}

We are now well-prepared to prove Proposition \ref{P1}.

\begin{proof}[Proofs of Proposition \ref{P1}]
As a consequence of Corollary \ref{C12} and the construction of $\hat H$, we find for any $x_0\in \partial \Omega$, $0<\tau<1<T$ and any $r>0$ small enough that
\[
\|\partial_t^k  h\|_{L^{\infty}((\tau,T)\times B_r(x_0))} \lesssim \|h_0\|_{L^2_{p+1}}.
\]
In particular, covering a small band along the domain boundary with a finite number of balls, the latter extends to the band $\Omega_r = \{x\in\Omega:\dist(x,\partial \Omega)\le r\}$ for some small $r>0$,
\[
\|\partial_t^k  h\|_{L^{\infty}((\tau,T)\times \Omega_r)} \lesssim \|h_0\|_{L^2_{p+1}}.
\]
As mentioned earlier, similar (but simpler, thanks to the strict parabolicity) arguments in the interior of the domain $\Omega$ yield analogous estimates on $\Omega\setminus \bar \Omega_r$. Both together prove the statement of the proposition.
\end{proof}

\section{Proof of {second dichotomy}} 
\label{S:second dichotomy}

In this final section, we turn to the proof of Theorem \ref{T3},  which states optimal exponential convergence of the relative error under the assumption that $V$ is an {ordinary limit} in the sense of Definition \ref{D3}. Thanks to our {first dichotomy result  --- Theorem \ref{thm-dichotomy} --- and Proposition \ref{P1}, it is {enough} to establish convergence at some} exponential rate, which is the main result of the present section.

\begin{theorem}[{Ordinary limits are approached exponentially fast}]\label{T4}
{Under the hypotheses of Theorem \ref{thm-dichotomy},}
 if $V$ is an ordinary limit of the dynamics \eqref{eq-U}, the convergence takes place exponentially fast, i.e., there exists a rate $\gamma>0$ such that
\bea \label{eq-thmapp}\Vert h(t)\Vert _{L^2_{p+1}} =O(e^{-\gamma t}) {\quad \mbox{\rm as} \quad t \to \infty.}
\eea 
\end{theorem}

The proof of exponential convergence relies  {on Choi and Sun's refinement --- Lemma \ref{lem-choisun} --- of the Merle-Zaag} dynamical systems result recalled above. It roughly says that if a solution is known to be small on a large time interval, then, up to a possible error caused by the neutral modes,  the stable and unstable modes should be much smaller than an exponentially decaying term in the middle of this time interval. (Note  the unstable modes tend to decay like the stable ones if time goes backward and this is why we need to go to the middle of time interval.) The main issue  to take care of  is thus the control of the neutral modes.

The underlying idea for controlling the neutral modes is  reducing the amplitude of the neutral modes by changing the reference stationary solution in the direction of the neutral modes. This can be effectively done if the {limit $V$ is ordinary}. 
This strategy goes back to the work of Allard and Almgren \cite{AA}, {who gave kernel integrability conditions guaranteeing that minimal surfaces converge to their tangent cones sequentially and exponenentially fast.}
See also Section 6 of {Simon \cite{MR821971}, or the recent contributions of Choi, Choi, Kim and Sun in various combinations \cite{ChoiSun20+} \cite{CCK}.}

In order to pursue this strategy, we have to prove that {being an ordinary limit is an open property among stationary solutions S}. This crucial insight requires some technical preparations.

\begin{lemma}[{Lower semicontinuity of kernel dimension at an ordinary limit}]
\label{L21}
Let {$V\in S$ be an ordinary limit} and $\delta>0$ as in Definition \ref{D3}. Let {$\tilde V \in S$ be sufficiently close that $h = {V}/{\tilde V}-1$ satisfies}
$\|h\|_{L^\infty} \le \delta$, so that     $h=\Phi_V(\psi)$ for some $\psi \in \ker L_V$. Then it holds that
\begin{equation}
\label{36b}
\frac{V}{\tilde V}(d\Phi_V)_\psi(\ker L_V ) \subset \ker L_{\tilde V}.
\end{equation}
\end{lemma}

\begin{proof}
{Let $0 \in \U \subset \ker L_V$ the neighbourhood provided by Definition \ref{D3}.}
We fix another element in the kernel, $\zeta \in \ker L_V$, and choose $s_0$ small enough such that  $\psi + s\zeta \in \U$ for any $s\in(-s_0,s_0)$. Then $h_s = \Phi_V(\psi+s\zeta)$ defines a family of stationary {relative errors} \eqref{36a} with $h_0 = h$, or equivalently, $\tilde V_s = V(h_s+1) \in S$ defines a stationary solution in terms of the original variables with $\tilde V_0=\tilde V$. Changing the reference stationary solution 
$\tilde h_s = \tilde V_s/\tilde V-1$ solves $\tilde h_0=0$ and
\[
L_{\tilde V}\tilde h_s = M(\tilde h_s).
\]
We may now rewrite
\[
\tilde h_s = \frac{V}{\tilde V}\left(\Phi_V(\psi+s\zeta) - \Phi_V(\psi)\right),
\]
and thanks to the regularity properties of the diffeomorphism $\Phi_V$, differentiation in the previous two identities yields
\begin{align*}
L_{\tilde V}  \left(\frac{V}{\tilde V}( d\Phi_V)_{\psi}\zeta\right)  = L_{\tilde V} \left. \partial_s \right|_{s=0} \tilde h_s = M'(0)  \left. \partial_s \right|_{s=0} \tilde h_s=0,
\end{align*}
{since $M(h)=O(h^2)$ as $h\to 0$} in \eqref{36a}. The latter verifies the inclusion \eqref{36b}.
\end{proof}

The next lemma guarantees that spectral gaps are preserved by nearby stationary solutions.

\begin{lemma}
[Continuity of spectral gap and nullity]\label{lem-continuity} 
Let {$V \in S$ be an ordinary limit and let the sequence $\{V_{\ell}\}_{\ell\in\N} \in S$ of stationary solutions
converge to $V$ relatively-uniformly, meaning $h_{\ell} = V_{\ell}/V-1$ satisfies}
\[
\|h_{\ell}\|_{L^{\infty}}\to 0\quad\mbox{as }\ell\to \infty.
\]
Let $\{\phi_{\ell}\}_{\ell\in\N}$ denote a sequence of normalized eigenfunctions, i.e., 
\[
L_{V_{\ell}}\phi_{\ell} = \lambda_{\ell}\phi_{\ell},\quad \mbox{and}\quad \int \phi_{\ell}^2 V_{\ell}^{p+1}\, dx  = 1,
\]
for some $\lambda_{\ell}\in\R$. Suppose that the sequence of eigenvalues is bounded, $|\lambda_{\ell}|\le \Lambda$ for some $\Lambda>0$. Then there exists a subsequence $\{\phi_{\ell_k}\}_{k\in\N}$ and a function $\phi\in L^2(V^{p+1}dx)$ such that
\[
\|\phi_{\ell_k}-\phi\|_{L^2(V^{p+1}dx)} \to 0\quad\mbox{as }k\to \infty.
\]
The limiting function $\phi$ is a normalized eigenfunction, i.e,
\[
L_{V}\phi = \lambda\phi,\quad \mbox{and}\quad \int \phi^2 V^{p+1}\, dx  = 1,
\]
for some $\lambda\in\R$. Moreover, the following hold true:
\begin{enumerate}
\item If $\lambda_{\ell}=0$ for all $\ell\in\N$, then $\lambda=0$.
\item If $\lambda_{\ell}>0$ for all $\ell\in\N$, then $\lambda>0$.
\item If $\lambda_{\ell}<0$ for all $\ell\in\N$, then $\lambda<0$.
\end{enumerate}
\end{lemma}
The lemma entails, in particular, that if $|\lambda_{\ell}|>0$ for all $\ell\in\N$, then 
\[
\liminf_{\ell\to\infty} |\lambda_{\ell}|  \ge \min\{ -\lambda_u, \lambda_s\},
\]
where $\lambda_u$ is the largest negative and $\lambda_s$ is the smallest positive eigenvalue of $L_V$.

\begin{proof}

For the compactness assertion, we aim to bound $\sup_\ell \|\phi_\ell \|_{H^1}$.
Applying {Proposition \ref{P5}} with $h/t =\phi_\ell=f/(\lambda_\ell t + 1)$  
yields the unweighted gradient bound
\[
\|\grad \phi_{\ell} \|_{L^2} \lesssim (\lambda_\ell+1) \|\phi_{\ell}\|_{L^2_{2p}} + \|\phi_{\ell}\|_{L^2_{p+1}} \lesssim (\lambda_\ell+2) \|\phi_{\ell}\|_{L^2_{p+1}}  \le  \Lambda+2.
\]
Now the local independence (in the relatively-uniform topology) of constants in $\lesssim$ on $V_\ell \in S$ combines with 
Lemma \ref{L13b} to imply the sequence $\phi_\ell$ is bounded in the Sobolev space $H^1(\Omega)$.
Via a Rellich compactness argument, we conclude that  $\{\phi_{\ell}\}_{\ell\in \N}$ converges strongly  subsequentially in $L^2$ (and then also in $L^2(V^{p+1}dx)$) and weakly in $H^1(V^2dx)$ towards some function $\phi$. 
 Moreover, thanks to the Bolzano--Weierstra\ss~theorem, we find that the sequence of eigenvalues $\{\lambda_{\ell}\}_{\ell\in\N}$ converges subsequentially to some $\lambda\in\R$. 
 
 Considering a common subsequence (indexed by $\ell_{k}$) and passing to the limit in the weak formulation of the eigenvalue equation,
 \[
 \int \grad \phi_{\ell_k}\cdot \grad f V_{\ell_k}^2\, dx  =\left(\lambda_{\ell_k} + p-1\right)\int \phi_{\ell_k} f V_{\ell_k}^{p+1}\, dx,
 \]
 where we also use the uniform convergence of the relative error $h_{k}$, we find that $\phi$ is a normalized eigenfunction of $L_V$ with {eigenvalue} $\lambda$.
 
 It remains to derive the assertion on the sign of the limiting eigenvalues. The first statement is trivial, while the proofs of two others are identical. Let's thus focus on one of them, say the {middle} one. It is clear that the limiting eigenvalue is nonnegative, $\lambda\ge 0$ and we have to rule out that it is in fact zero. For this purpose, we note that the eigenfunctions $\phi_{\ell}$ are orthogonal to the kernel,
 \begin{equation}
 \label{36c}
 \int \phi_{\ell} \zeta_{\ell} V_{\ell}^{p+1}\, dx = 0\quad \mbox{for any }\zeta_{\ell}\in \ker L_{V_\ell}.
 \end{equation}
We pick $\zeta \in \ker L_V$ and define $\zeta_{\ell}\in \ker L_{V_{\ell}}$ according to the inclusion \eqref{36b} derived in Lemma \ref{L21} by
\[
\zeta_{\ell}  = \frac1{h_{\ell}+1}(d\Phi_V)_{\psi_{\ell}} \zeta,
\]
where $\psi_{\ell}\in \ker L_V$ is such that $h_{\ell} = \Phi_V(\psi_{\ell})$.  It is straightforward to verify that $\zeta_{\ell}$ converges to $\zeta$ strongly in $L^2(V^{p+1}dx)$ as $\ell\to\infty$. Indeed, because of the imposed convergence of the relative error $h_{\ell}$ and since the diffeomorphism $\Phi_V$ vanishes only at the origin, we must have that $\psi_{\ell}\to0$ strongly in $L^2(V^{p+1})$. Furthermore, by the continuity of the derivative $d\Phi_V$, it holds that $(d\Phi_V)_{\psi_{\ell}}\to \mathrm{id}$. Using once again the uniform convergenve of the relative error, we conclude that $\zeta_{\ell}\to \zeta$ in $L^2(V^{p+1}dx)$. 

We now pass to the limit in the orthogonality condition \eqref{36c} with our particular construction of the $\zeta_{\ell}$'s, which was arbitrary in the choice of $\zeta$, and find
\[
\int \phi\zeta V^{p+1}\, dx  =0\quad \mbox{for any }\zeta \in \ker L_V.
\]
Hence, $\phi$ is a (nontrivial) eigenfunction of $L_V$ that is orthogonal to the kernel. We conclude that $\lambda>0$ as desired.
\end{proof}
The preceeding analysis allows us to conlude quite easily that the dimension of the kernels of the linear operators remains constant if the reference stationary solution is changed in a neighborhood of {an ordinary limit} $V$.

\begin{lemma}[Invariance of the kernel dimension near ordinary limits]
\label{L23a}
Let {$V\in S$ be an ordinary limit and  $\delta>0$ as in Definition \ref{D3}. Let $\tilde V\in S$ be a stationary solution 
close to $V$ in the sense that $h = \tilde V/V-1$ satisfies} $\|h\|_{L^{\infty}}\le \tilde \delta$,
for some $\tilde \delta\in(0,\delta)$. Then {$\tilde \delta$ sufficiently small implies}
\begin{equation}\label{64}
 \dim \ker L_V  = \dim \ker L_{\tilde V}.
\end{equation}
\end{lemma}

\begin{proof}
As a consequence of Lemma \ref{L21}, {and} because $(d\Phi_V)_{\psi}$ is an isomorphism, it is clear that
\begin{equation}\label{64a}
K= \dim \ker L_V \le \dim \ker L_{\tilde V}.
\end{equation}
We argue that both kernels have indeed the same dimension if $\tilde\delta $ is sufficiently small. We give an indirect argument and {derive} a contradiction {by assuming} that there exists a sequence $ h_{\ell} =  V_{\ell}/V-1$ satisfying $\| h_{\ell}\|_{L^{\infty}}\le \frac1{\ell}$ and $\dim \ker L_{ V_{\ell}}\ge K+1$. We pick $K+1$ orthonormal (and thus linearly independent) functions $\phi_{\ell,1},\dots,\phi_{\ell,K+1}$ in $\ker L_{ V_{\ell}}$. 

By the virtue of Lemma \ref{lem-continuity}, there exist subsequences (that we will not relabel) and normalized functions $\phi_1,\dots,\phi_{K+1}$ in $\ker L_V$ such that
\[
\phi_{\ell,k}\to \phi_k\quad \mbox{in }L^2(V^{p+1})\mbox{ as }\ell\to\infty.
\]
In particular,   using in addition that $\{h_{\ell}\}_{\ell\in\N}$ is converging uniformly, we find  that $\phi_1,\dots,\phi_{K+1}$ is  orthonormal. 
Thus, the dimension of $\ker L_V$ is at least  $K+1$, which contradicts~\eqref{64a}. 
\end{proof}

Collecting these technical preparations, we are able to derive the aforementioned key feature: 
{that ordinariness is a relatively-uniformly open property in the set $S$ of stationary limits.}

\begin{proposition}[{Ordinary limits form an open subset of $S$}]
 \label{L23}
Let $V \in S$ be an ordinary limit. {If $\tilde V \in S$ is a stationary solution sufficiently close
to $V$ in the relatively-uniform topology,} then $\tilde V$ is also {an ordinary limit}.

\end{proposition}

\begin{proof} By combining the results from Lemmas \ref{L21} and \ref{L23a}, we see that
\[
 \frac{\tilde V}{V} \ker L_{\tilde V} = (d \Phi_V)_\psi(\ker L_V),
 \]
 where $\psi\in \ker L_V$ is such that $h = \Phi_V(\psi)$. 
It follows that the mapping $\Phi_{\tilde V}$ defined by 
\[ \Phi_{\tilde V}(\tilde \psi) :=\frac{V}{\tilde V} \left[\Phi_V\left (\psi+ ((d\Phi_V)_\psi)^{-1} (\tfrac{\tilde V}{V} \tilde \psi) \right) -\Phi _V(\psi) \right],\] for $\tilde \psi \in \ker {L_{\tilde V}}$, is a $C^1$ diffeomorphism from a neighborhood of   $0\in\ker L_{\tilde V}$ into the set of stationary solutions $L_{\tilde V}\tilde h = M(\tilde h)$ and satisfies all the properties listed in Definition \ref{D3}, as can be readily verfied. Notice also that  we have see the construction of this diffeomorphism already in the proof of Lemma \ref{L21}: If $\zeta\in\ker L_V$ is sufficiently small so that  $h_{\zeta} = \Phi_V(\psi + \zeta)$ is well-defined, then $V_{\zeta} = V(h_{\zeta}+1)$ is a stationary solution and the relative error $\tilde h_{\zeta}$ with respect to $\tilde V$, 
\[
\tilde h_{\zeta} = \frac{V_{\zeta}}{\tilde V} - 1 = \frac{V}{\tilde V}\left(h_{\zeta} - \frac{\tilde V}{V}+1\right) = \frac{V}{\tilde V}\left(\Phi_V(\psi+\zeta) - \Phi_V(\psi)\right),
\]
solves the stationary error equation relative to the reference point $\tilde V$.
 \end{proof}

We will now derive a building block for the proof of Theorem \ref{T4} which exploits  both the previous Proposition \ref{L23} and the {Choi-Sun refinement of the} Merle--Zaag Lemma \ref{lem-choisun}, {to yield} a suitable reduction of the neutral modes as announced at the beginning of this section. The actual proof of exponential convergence will  follow subsequently by iteration.

\begin{proposition}[Improvement by changing {reference} stationary solutions]\label{P24-2} 
 Let $V\in S$ be an {ordinary limit} \eqref{eq-V}. There exist four constants $\e_0,\delta_0\in(0,1)$ and $\tau,C\in(1,\infty)$ with the following property: Let $V(1+h(t))$ {solve the dynamics \eqref{eq-U} and stay} near $V$,  \begin{equation}\label{24-1}
 \sup_{t\ge 0} \|h(t)\|_{L^{\infty}}\le  \delta_0;
 \end{equation}
$\{V(1+g_s)\}_{s\ge 0} {\subset S}$ be a family of stationary solutions to \eqref{eq-V} {also} close to $V$,   
 \begin{equation}\label{24-2} 
 \sup_{s\ge 0} \Vert g_s  \Vert_{L^{\infty}} \le  \delta_0;
  \end{equation} and  suppose $V(1+h(t))$ is $\e$-close to $V(1+g_s)$ on $t\in [s,s+2]$, 
 \begin{equation}\label{24-3}
      \sup_{t\in[s,s+2]}\Vert h(t) -g_s \Vert _{L^2_{p+1}} \le \eps 
 \end{equation}
  for all $s\ge0$ for some $\e \le \e_0$.

Then there is another family of stationary solutions $\{V(1+\hat g_s)\}_{s\ge \tau}$ so that   $V(1+h(t))$ is $\tfrac\e2$-close to $V(1+\hat g_s)$ on $t\in [s,s+2]$ 
 \begin{equation}\label{24-4}
      \sup_{t\in[s,s+2]}\Vert h(t) -\hat g_s \Vert _{L^2_{p+1}} \le \frac \eps 2 
 \end{equation}
  for all $s\ge\tau $. Moreover, for each $s\ge0$, $\hat g_{s+\tau}$ is close {enough to $g_s$  that}
 \begin{equation}\label{24-5}
     \Vert g_s -\hat g_{s+\tau} \Vert_{L^\infty} \le C \eps . 
 \end{equation}
  \end{proposition}
In the proof of Theorem \ref{T4} below, the exponential convergence rate $\gamma ={\frac{\log 2}{\tau}} $ is determined by the delay $\tau$ provided by the preceding {proposition}.

\begin{proof}To simplify the notation, we write 
\[\| f\| _{\tilde V} = \| f \|_{L^2(\tilde V^{p+1}) }\] for any $\tilde V$  stationary solution  \eqref{eq-V}. We start by noting that for two stationary solutions $\tilde V = V(1+\tilde g)$ and $\hat V = V(1+\hat g)$ to \eqref{eq-V} with $\Vert \tilde g \Vert _{L^\infty} $, $\Vert \hat g \Vert _{L^\infty}\le \delta_0<1$, there holds 
\[ \Vert f \Vert _{\tilde V} \le  \left(\frac{1+ \delta_0}{1-\delta_0}\right)^{\frac{p+1}2}  \|f\|_{\hat V}.\] 
 i.e.,  two norms are equivalent 
\[\Vert f \Vert _{\tilde V } \lesssim \Vert f \Vert_{\hat V}  \]
provided $\delta_0<1/2$. Throughout this proof, $f \lesssim g$ denotes the inequality $f\le C g$ for some constant $C$ which may depend on $n$, $m$, and $V$ but uniform in small $\e_0$, $\delta_0$, and  large $\tau$.

It suffices to show how $\hat g _\tau$ is chosen so to satisfy \eqref{24-4} and \eqref{24-5}. For other $\hat g_{s'+\tau}$, $s'>0$, we may shift the time of original problem by $s'$, consider $V(1+h(s'+t))$ and $V(1+g_{s'+s})$ in place of $V(1+h(t))$ and $V(1+g_s)$, respectively, and re-apply the previous assertion.  We now notice that by the triangle inequality and the hypothesis in \eqref{24-3}, we have that
\[
\|g_{s+2}-g_{s}\|_V \le \|g_{s+2}-h_{s+1}\|_V + \|h_{s+1}-g_{s}\|_V \le 2\eps,
\]
 for any $s\ge 0$, and thus by iteration, 
 \begin{equation}\label{24-6}
 \|g_{2k} - g_0\|_V \le k\eps,
 \end{equation}
 for any $k\in \N_0$.
 By another application of the triangle inequality on \eqref{24-3} and \eqref{24-6}, it holds that, for $\tau >1$,
 \be \label{24-7} 
  \sup_{t \in [0,2\tau]}\|h(t)-g_0\|_V \lesssim  \tau \eps.\ee 

Let us denote $\tilde V := V(1+g_0)$, a new stationary solution. If we write the solution $V(1+h(t))$ to \eqref{eq-U} in terms of $\tilde V$ and its relative quantity by $V(1+h(t))= \tilde V(1+\tilde h(t))$, then $\tilde h(t) := \frac{h(t)-g_0}{1+g_0}$.  Note that $\tilde h(t)$ solves the evolution equation
for relative error \begin{equation}\label{24-9}
 \partial_t \tilde h + L_{\tilde V}\tilde h = N_{\tilde V}(\tilde h),
 \end{equation}
equation \eqref{eq-relativeerror} with new {refenence} stationary solution $\tilde V$. 

In view of \eqref{24-1}, \eqref{24-2}, \eqref{24-7}, and the equivalence between $\Vert\cdot  \Vert_{V}$ and $\Vert\cdot  \Vert_{\tilde V}$, observe that 
\be  \label{24-10} \sup _{t\ge 0}\Vert \tilde h(t) \Vert _{L^\infty} \le \frac{2\delta_0 }{1-\delta_0}\quad \text{ and } \quad \sup_{t\in[0,2\tau] }\Vert \tilde h(t)\Vert_{\tilde V} \lesssim \tau \eps   . \ee 
	If we choose $\delta_0$ sufficiently small, then the smoothing estimate in Proposition \ref{P1} applies to $\tilde h(t)$, a solution to \eqref{24-9}, and we have 
	\be  \label{24-10-1}\sup_{t\in[1,2\tau] }\Vert \tilde h(t)\Vert_{L^\infty}+ \Vert \p_t \tilde h(t) \Vert_{L^\infty} \lesssim \e \tau  .\ee
	(In fact, the constants $\e$ and $C$ in Proposition \ref{P1} depend on $\tilde V$, not $V$. Since $\tilde V=  V(1+g_0)$ and $\Vert g_0 \Vert_{L^\infty}\le \delta_0 $, by assuming $\delta_0$ is small, we may assume $\e$ and $C$ do not depend on the choice of $g_0$.)

In the next step, we aim at applying a Merle--Zaag-type lemma to the unstable, center, and stable modes of $\tilde h$. For this purpose, we introduce the $L^2(\tilde V^{p+1}dx)$ orthogonal projections $\tilde P_u$, $\tilde P_u$ and $\tilde P_s$ onto the unstable, center, and stable eigenspaces generated by $L_{\tilde V}$, and write $\tilde h_u = \tilde P_u\tilde h$, $\tilde h_c = \tilde P_c\tilde h$ and $\tilde h_s = \tilde P_s\tilde h$. Arguing similarly {to} the proof of Theorem \ref{thm-dichotomy}, we derive the system of ordinary differential equations
\begin{align*}
\frac{d}{dt} \|\tilde h_u\|_{\tilde V}+ \tilde \lambda_u \|\tilde h_u\|_{\tilde V} &\ge -C\eps \tau \left(\|\tilde h_u\|_{\tilde V}+ \|\tilde h_c\|_{\tilde V} + \|\tilde h_s\|_{\tilde V}\right),\\
\left|\frac{d}{dt} \|\tilde h_c\|_{\tilde V}  \right| &\le C\eps \tau \left(\|\tilde h_u\|_{\tilde V}+ \|\tilde h_c\|_{\tilde V} + \|\tilde h_s\|_{\tilde V}\right),\\
\frac{d}{dt} \|\tilde h_u\|_{\tilde V}+ \tilde \lambda_s \|\tilde h_u\|_{\tilde V} &\le C\eps \tau \left(\|\tilde h_u\|_{\tilde V}+ \|\tilde h_c\|_{\tilde V} + \|\tilde h_s\|_{\tilde V}\right),
\end{align*}
for all $t\in [1,2\tau]$, for some constant $C>0$, where $\tilde \lambda_u$ is the largest negative and $\tilde \lambda_s$ the smallest positive eigenvalue of $L_{\tilde V}$. Note that $\e_0>0$ and $\tau>0$ are not fixed yet. Let us denote $\sigma:=C\e \tau$. Suppose we choose $\e_0$ and $\tau$ so that 	     \[ \sigma \le C \e_0 \tau \le \tilde  \lambda := \frac 12 \min\{-\tilde \lambda_u,\tilde \lambda_s\},\]
	     then the system implies
\begin{align*}
\frac{d}{dt} \|\tilde h_u\|_{\tilde V}- \tilde \lambda \|\tilde h_u\|_{\tilde V} &\ge -\sigma \left( \|\tilde h_c\|_{\tilde V} + \|\tilde h_s\|_{\tilde V}\right),\\
\left|\frac{d}{dt} \|\tilde h_c\|_{\tilde V}  \right| &\le \sigma \left(\|\tilde h_u\|_{\tilde V}+ \|\tilde h_c\|_{\tilde V} + \|\tilde h_s\|_{\tilde V}\right),\\
\frac{d}{dt} \|\tilde h_u\|_{\tilde V}+ \tilde \lambda \|\tilde h_u\|_{\tilde V} &\le\sigma \left(\|\tilde h_u\|_{\tilde V}+ \|\tilde h_c\|_{\tilde V} \right),
\end{align*}
for any $t\in [1,2\tau]$. 
	By the virtue of Lemma \ref{lem-choisun} and the bound in \eqref{24-10}, there exits a constant $\sigma_0$ dependent only on $\tilde \lambda$ such that if we further assume $ C\e_0 \tau \le \sigma_0$, it holds that
\begin{equation}\label{24-11}
\|\tilde h_u(t)\|_{\tilde V} + \|\tilde h_s(t)\|_{\tilde V} \lesssim \eps \tau \|\tilde h_c(t)\|_{\tilde V} + \eps \tau  e^{-\frac18\tilde \lambda \tau}\lesssim (\eps \tau)^2 + \eps \tau  e^{-\frac18\tilde \lambda \tau},
\end{equation}
for any $t\in[\frac {3\tau}4,\frac{5\tau}{4}]$. Here, we applied the lemma on the interval $t\in [ \frac{\tau }{2} ,\frac {3\tau}{2} ]$. We remark that as a direct consequence of Lemma \ref{lem-continuity}, $\tilde \lambda$ can be bounded away from zero uniformly in $\tilde V$, more precisely, we can suppose that
\begin{equation}
\label{24-11a}
\tilde \lambda > \frac 12 \min \{-\lambda_u, \lambda_s\} 
\end{equation}
if $\delta_0$ is chosen sufficiently small.

 We still have to bound the center modes. For this, we make use of the fact that, by the virtue of Proposition \ref{L23} and our hypothesis in \eqref{24-2},  the {limit} $\tilde V = V(g_s+1)$ is {ordinary}. We denote by $\Phi_{\tilde V}$ the diffeomorphism between suitable subsets of $\ker L_{\tilde V}$ and the set of stationary solutions relative to $\tilde V$
	    described in Definition~\ref{D3}. Thanks to the bound in \eqref{24-10} that $
 \|\tilde h_c(\tau)\|_{\tilde V} \le \|\tilde h(\tau )\|_{\tilde V} \lesssim \eps \tau$,  for any small $\hat \delta$, if $\eps_0 \tau$ is chosen sufficiently small, there exists a function $\tilde g_\tau = \Phi_{\tilde V}(\tilde h_c(\tau))$ with $\|\tilde g_\tau\|_{L^{\infty}}\le \hat \delta$ solving $L_{\tilde V}(\tilde g_\tau) = M(\tilde g_\tau)$ (i.e., $\tilde V(1+\tilde g_\tau)$ solves \eqref{eq-V}.) Moreover, since $\Phi_{\tilde V}(0)=0$ and $(d\Phi_{\tilde V})_0=\mathrm{id}$, we have that
\bea\label{24-12}
\|\tilde g_\tau - \tilde h_c(\tau)\|_{\tilde V} &= \|\Phi_{\tilde V}(\tilde h_c(\tau)) - \Phi_{\tilde V}(0) - (d\Phi_{\tilde V})_0(\tilde h_c(\tau))\|_{\tilde V} \\& = o(\|\tilde h_c(\tau)\|_{\tilde V}) = o(\eps \tau).
\eea 
 In order to observe that this estimate is stable under-order-one variations in time, we recall that the center modes
solve the (finite dimensional) system $\partial_t \tilde h_c = \tilde P_c N_{\tilde V}(\tilde h)$. An integration over some interval $[\tau ,t]$ and the quadratic estimate \eqref{51} on the nonlinearity give
\[
\|\tilde h_c(t) -\tilde h_c(\tau )\|_{\tilde V} \le \int_{\tau }^t \|N_{\tilde V}(\tilde h)\|_{\tilde V}\, dt \lesssim\int_{\tau }^t\left(\|\tilde h\|_{L^{\infty}} + \|\partial_t \tilde h\|_{L^{\infty}}\right) \|\tilde h\|_{\tilde V}\, dt.
\]
 We apply the bound in \eqref{24-10} and \eqref{24-10-1} to conclude
 \[
 \sup_{ t\in [\tau,\tau+2 ]} \|\tilde h_c(t)-\tilde h_c(\tau )\|_{\tilde V} \lesssim (\eps \tau )^2
 ,\] and thus, \eqref{24-12} can be generalized to
 \begin{equation}\label{24-13}
\sup_{t\in[\tau ,\tau+2]}\|\tilde g_{\tau} - \tilde h_c(t)\|_{\tilde V}   = o(\eps \tau).
\end{equation}

Later when we show \eqref{24-5}, it will be necessary to quantify the relation between $\hat \delta$ and $\eps \tau$. For this purpose, we notice that since $\tilde g_\tau $ solves $L_{\tilde V}\tilde g_\tau  = M(\tilde g_\tau )$, it satisfies the estimate $\|\tilde g_\tau \|_{L^{\infty}} \lesssim \|\tilde g_\tau \|_{\tilde V}$, and the inequality is uniform in $\tilde V$ and depends only on the bound in \eqref{24-2}. Indeed, this estimate can be derived parallel to the smoothing estimates in Proposition \ref{P1}. Therefore, using the properties of the diffeomorphism $\Phi_{\tilde V}$ again, we observe that
\begin{equation}\label{24-14}
\|\tilde g_\tau \|_{L^{\infty}} \lesssim \|\tilde g_\tau \|_{\tilde V}  = \|\Phi_{\tilde V}(\tilde h_c(\tau ))\|_{\tilde V} \lesssim \eps \tau
.\end{equation}

Let us summarize what we have obtained so far. We showed there is small $\delta_0>0$ and $c_0>0$ such that if $\e_0 \tau <c_0$ and $\e <\e_0$, then  there is a stationary solution $\tilde V(1+\tilde g_\tau)$ with the estimates \eqref{24-11}, \eqref{24-12}, \eqref{24-13}, and \eqref{24-14}. We shall now transform $\tilde g_\tau $ into the  solution  $\hat g_\tau $ of \eqref{36a} that we are looking for. We thus set \be \label{24-15} \hat g_\tau  = (g_0+1)\tilde g_\tau  + g_0.\ee Recalling that $h(t)$ and $\tilde h(t)$ are related by the same transformation, $h(t) = (g_0+1)\tilde h(t)+g_0$, and using the estimates \eqref{24-2}, we have that
\[
\|h(t) - \hat g_r\|_V  = \|(g_s+1)(\tilde h(t) - \tilde g_r)\|_V \lesssim \|\tilde h(t) - \tilde g_r\|_{\tilde V}.
\]
It remains to use the estimates in \eqref{24-11} and \eqref{24-13} together with the triangle inequality to deduce\[
\sup_{t\in[\tau ,\tau +2]}\|h(t) - \hat g_\tau \|_V \lesssim (\eps \tau)^2 + \eps \tau  e^{-\frac18\tilde \lambda \tau} + o(\eps \tau)
.\] 
First choose  $\tau $ large (independently from $\tilde V$ thanks to \eqref{24-11a}) and then $\eps_0$ sufficiently small to satisfy $\e _0\tau<c_0$ and   \eqref{24-4}
\[\sup_{t\in[\tau,\tau+2]} \Vert h(t)-\hat g_\tau\Vert_V \le \frac{1}{2} \eps \]for any $\eps\le \eps_0$.

Finally, \eqref{24-5} is an immediate consequence of \eqref{24-14} and the definition of $\hat g_\tau $ in \eqref{24-15}. 
\end{proof}

We now have  all the tools at hand to proceed to the proof of the exponential convergence result.
	\begin{proof}[Proof of Theorem \ref{T4}] The idea is to iterate Proposition \ref{P24-2}. By a translation in time, we may assume 
	\[
	\sup_{t\ge 0}\Vert h(t) \Vert_{L^\infty}\le \e_1
	\]
	 for some $\eps_1  \le \min(\delta_0,\e_0)$. Moreover, we choose $\e_1$ small so that $2C\e_1 <\delta_0$ holds, where $C<\infty$ is the constant in Proposition~\ref{P24-2}. 
	
	To start with, we consider the trivial solution family $g_{0,s}:=0$ for $s\ge 0$, which trivially satisfies the hypothesis of  Proposition \ref{P24-2} with $\eps=\eps_1$. Therefore  there exists a family of stationary errors $\{\hat g_{0,s}\}_{s\ge\tau}$ satisfying
	\[
	\sup_{s\ge \tau}\Vert \hat g_{0,s}\Vert_{L^\infty} \le C \e_1
	\]
	 and 
	 \[
	 \sup_{t\in [s,s+2]} \Vert h(t)- \hat g_{0,s} \Vert_{L^2_{p+1} }\le \frac{\e_1}2 
	 \]
	  for all $s\ge \tau$. Hence, the translated family $\{g_{1,s}\}_{s\ge 0}$ where $g_{1,s} = \hat g_{0,s+\tau}$ satisfies again the hypothesis of    Proposition \ref{P24-2}, this time with $\eps = \eps_1/2$ and $h(t)$ replaced by $h(t+\tau)$. We perform a series of iterations, leading to families $\{g_{k,s}\}_{s\ge 0}$ for any $k\in\N$ satisfying
\begin{equation}
\label{66}
\sup_{t\in [s,s+2]} \Vert h(t+k\tau )- g_{k,s} \Vert_{L^2_{p+1} }\  \le \frac{\e_1}{2^k}
\end{equation}
 and 
 \begin{equation}\label{65}
 \Vert g_{k-1,s}- g_{k,s }\Vert_{L^\infty} \le \frac{C\e_1}{2^{k-1}} 
 \end{equation}
  for all $s\ge 0$. Notice that the latter and the fact that we started with the trivial solution $g_{0,s} = 0$ entails that
  	\[\Vert g_{k,s}  \Vert_{L^\infty} \le \sum _{\ell=1}^k \Vert g_{k-\ell+1,s} - g_{k-\ell,s} \Vert_{L^\infty}\le  2C\e_1 \sum_{\ell=1}^k  2^{-k} \le  2C\e_1< \delta_0,\]
	by our choice of $\eps_1$, which guarantees that condition  \eqref{24-2} holds true in every iteration step.

	Moreover, the estimate \eqref{65} also implies that, for any fixed $s$, the sequence   $g_{k,  s}$ converges geometrically in $L^\infty$ to some $g_{\infty, s} $ ,
	\[
	\|g_{k,s} - g_{\infty,s}\|_{L^{\infty}} \le \frac{C\eps_1}{2^k},
	\]
for any $k\in\N$ and $s\ge 0$. We conclude via  the triangle inequality and estimate \eqref{66} that
\[
\|h(t+k\tau) - g_{\infty,s}\|_{L^2_{p+1} }\ \le C_1 2^{-k}
\]
for some new constant $C_1$,   any $s\ge 0$ and any $t\in[s,s+2]$. Picking $s=t=0$, we deduce exponential convergence with rate $\gamma = (\log 2)/\tau$ towards $g_{\infty,0}$, which must actually vanish,    $g_{\infty,0}=0$, because $h(t)$ is decaying to zero by the virtue of the Bonforte--Grillo--V\'azquez theorem \cite{BonforteGrilloVazquez12}, cf.~\eqref{vre}. This finishes the proof of the {second dichotomy}.
\end{proof}


\bibliography{fastdiffusion}
\bibliographystyle{alpha}
\end{document}